\documentclass[10pt]{amsart}
\pdfoutput=1

\usepackage[a4paper,bottom=35mm]{geometry}
\usepackage[OT2,OT1]{fontenc}
\usepackage{amsmath,amssymb}
\usepackage[UKenglish]{babel}
\usepackage{enumerate,mathrsfs,mathtools}
\usepackage{graphicx,tikz,tikz-cd}
    \graphicspath{ {./gfx/} }
\usepackage{caption}
\usepackage[colorlinks=true,urlcolor=blue,linkcolor=blue,citecolor=blue]{hyperref}

    \newtheorem{theorem}{Theorem}
    \newtheorem{proposition}[theorem]{Proposition}
    \newtheorem{corollary}[theorem]{Corollary}
    \newtheorem{lemma}[theorem]{Lemma}
\theoremstyle{definition}
    \newtheorem{definition}[theorem]{Definition}

    \newtheorem{remark}[theorem]{Remark}
    \newtheorem{ruledup}[theorem]{Rule}
\numberwithin{equation}{section}
\numberwithin{theorem}{section}
\numberwithin{figure}{section}

\newcommand\switchToCyrilic{%
    \renewcommand\rmdefault{wncyr}%
    \renewcommand\encodingdefault{OT2}%
    \normalfont\selectfont}
\DeclareTextFontCommand{\textcyr}{\switchToCyrilic}

\DeclareMathOperator{\Hom}{Hom}
\DeclareMathOperator{\Imag}{Im}

\DeclareMathOperator{\inc}{inc}

\DeclareMathOperator{\Ker}{Ker}
\DeclareMathOperator{\Obs}{Obs}
\DeclareMathOperator{\linspan}{span}

\DeclareMathOperator{\parity}{\Pi}

\newcommand{\ahol}[1]{\measuredangle_{#1}}

\newcommand{\AFpar}{\mathcal{A}^{\FF_2}_{\mathrm{even}}}
\newcommand{\Areal}{\mathcal{A}^{\RR}}

\newcommand{\CC}{\mathbb{C}}

\newcommand{\del}{\partial}

\newcommand{\edges}{\mathcal{E}}

\newcommand{\FF}{\mathbb{Z}}
\newcommand{\geo}{{\mathrm{geo}}}

\newcommand{\intersection}{\iota}

\newcommand{\maillink}[1]{\href{mailto:#1}{#1}}
\newcommand{\Mat}{\mathcal{M}}

\newcommand{\nword}[2]{\ensuremath{#1}\nobreakdash-#2}

\newcommand{\PD}{\mathrm{{P}\mkern-1mu{D}}}

\newcommand{\Proofpart}[2]{\par\noindent\textbf{#1~#2}\ }
\newcommand{\PSLC}{PSL(2,\CC)}
\newcommand{\PSLR}{PSL(2,\RR)}

\newcommand{\rect}{\mathrm{rect}}
\newcommand{\RR}{\mathbb{R}}
\newcommand{\SAS}{\mathcal{S\mkern-2.5mu A}}
\newcommand{\SASdistinguished}{\mathcal{S\mkern-2.5mu A}^\geo}

\newcommand{\SLC}{SL(2,\CC)}

\newcommand{\Sval}{\nword{S^1}{valued}}
\newcommand{\TAS}{T\mkern-2mu\mathcal{AS}}
\newcommand{\tet}{\Delta}

\newcommand{\transp}{\top} % Transpose symbol; it always goes in the superscript
\newcommand{\triang}{\mathcal{T}}

\newcommand{\ZZ}{\mathbb{Z}}

\begin{document}

% Define the title:
\title[Circle-valued angle structures]{Circle-valued angle structures and obstruction theory}

% Define authors
\author[C. Hodgson]{Craig D. Hodgson}
\author[A. Kricker]{Andrew J. Kricker}
\author[R. Siejakowski]{Rafa\l\ M. Siejakowski}

% MSC classification and keywords:

% Addresses of authors and their emails:
\address{% Craig's address
School of Mathematics and Statistics\endgraf
The University of Melbourne\endgraf
Parkville, Victoria 3010\endgraf
Australia
}
\email{\maillink{craigdh@unimelb.edu.au}}
\address{% Andrew's address
School of Physical and Mathematical Sciences\endgraf
Nanyang Technological University\endgraf
21~Nanyang Link\endgraf
Singapore~637371
}
\email{\maillink{ajkricker@ntu.edu.sg}}
\address{% Rafael's address
Instituto de Matem\'atica e Estat\'\i{}stica\endgraf
Universidade de S\~ao Paulo\endgraf
Rua do Mat\~ao 1010, 05508-090 S\~ao Paulo, SP\endgraf
Brazil
}
\email{\maillink{rafal@ime.usp.br}}

% The abstract of the paper
\begin{abstract}
We study spaces of circle-valued angle structures, introduced by Feng Luo, on ideal triangulations of 3-manifolds. We prove that the connected components of these spaces are enumerated by certain cohomology groups of the 3-manifold with $\mathbb{Z}_2$-coefficients. Our main theorem shows that this establishes a geometrically natural bijection between the connected components of the spaces of circle-valued angle structures and the obstruction classes to lifting boundary-parabolic $\PSLC$-representations of the fundamental group of the 3-manifold to boundary-unipotent representations into $\SLC$.  
In particular, these connected components have topological and algebraic significance independent of
the ideal triangulations chosen to construct them.
The motivation and main application of this study is to understand the domain of the state-integral defining the meromorphic 3D-index of Garoufalidis and Kashaev, necessary in order to classify the boundary-parabolic representations contributing to its asymptotics.
\end{abstract}

% Print the title and TOC:
\maketitle
\tableofcontents

%==============================================================================
% Section 1: Introduction
%==============================================================================
\section{Introduction}

The subject of this article is a type of structure on an ideal triangulation of a three-manifold, called an \emph{\Sval\ angle structure}, or alternatively a \emph{circle-valued angle structure}.
These structures were first introduced by Luo \cite{luo-volume}, who explored the connections between
normal surface theory and Thurston's gluing equations on ideal triangulations.

We begin by introducing some notation that we will use to discuss ideal triangulations. Then we will define \Sval\ angle structures and introduce the main results of this paper.

Throughout the paper, we assume that $M$ is a compact, connected, orientable \nword{3}{manifold}~$M$ whose boundary~$\del{M}$
consists of $k\geq1$ tori. 

\begin{figure}[tbh]
    \centering
    \begin{tikzpicture}
        \node[anchor=south west,inner sep=0] (image) at (0,0)
            {\includegraphics[width=0.7\columnwidth]{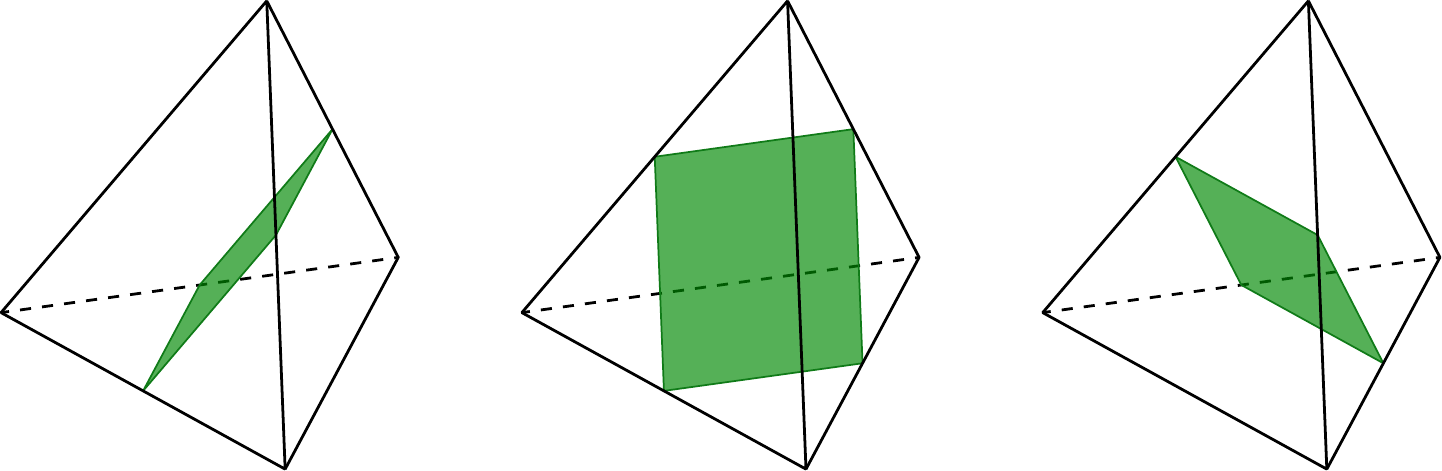}};
        \begin{scope}[x={(image.south east)},y={(image.north west)},
                every node/.style={inner sep=0, outer sep=0}]
            \node[anchor=south east] at (0.1, 0.8)  {$\square$};
            \node[anchor=south east] at (0.47, 0.8)  {$\square'$};
            \node[anchor=south east] at (0.84, 0.8) {$\square''$};
        \end{scope}
    \end{tikzpicture}
    \caption{%
    The cyclic ordering of normal quadrilaterals $\square\to\square'\to\square''\to\square$.
    }\label{fig:quads}
\end{figure}
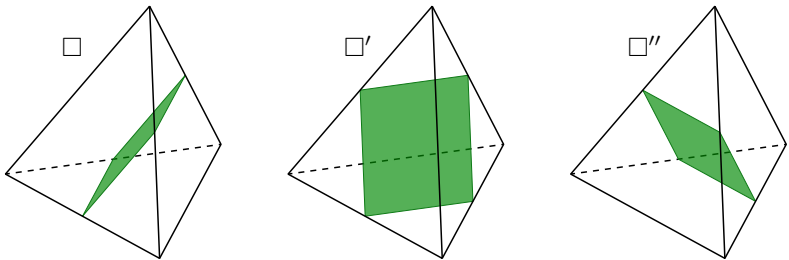
Following \cite{thurston-notes,neumann-zagier}, we understand an \emph{ideal triangulation} $\triang$ of $M$ to be a
collection $\{\tet_1,\dotsc,\tet_N\}$ of $N\geq2$ distinct, oriented topological ideal tetrahedra
equipped with a system of
face pairings, with the property that the resulting quotient space is homeomorphic to the interior of $M$. 
The face pairings are orientation-reversing 
homeomorphisms identifying pairs of distinct faces of the tetrahedra~$\tet_1,\dotsc,\tet_N$
in such a way that no face remains unglued.
In particular, we allow identifications between different faces of the same tetrahedron, as well
as multiple face gluings between a pair of tetrahedra. 

When $\triang$ is an ideal triangulation,
we denote by $Q(\triang)$ the set of \emph{normal quadrilateral types} in $\triang$,
and by $\edges(\triang)$ the set of edges of $\triang$.
Recall that there are three normal quadrilateral types in each tetrahedron, shown in Figure~\ref{fig:quads}, and so
$|Q(\triang)|=3N$.
Since all connected components of $\partial M$ are tori,
a simple Euler characteristic calculation \cite{youngchoi} shows that $|\edges(\triang)|=N$.

We will use the term \emph{edges} to refer to the elements of $\edges(\triang)$
and the term \emph{tetrahedral edges} when speaking about the edges of the unglued tetrahedra.
The gluing pattern of $\triang$ induces identifications among the tetrahedral edges,
so that each edge~$E\in\edges(\triang)$
can be thought of as an equivalence class of tetrahedral edges.

The three normal quads
contained in an oriented tetrahedron~$\tet$ admit a cyclic ordering
$\square\to\square'\to\square''\to\square$ depicted in Figure~\ref{fig:quads};
note that reversing the orientation of $\tet$ reverses this ordering.
We shall write $\square\subset\tet$ when
$\square$ is a normal quadrilateral type in a tetrahedron~$\tet$, and then the symbols $\square'$ and $\square''$ will denote the quad types that follow $\square$ in the cyclic ordering.

Every normal quad~$\square\subset\tet$ determines a pair of opposite tetrahedral edges of $\tet$,
namely the two edges \emph{facing} the quadrilateral~$\square$ (and disjoint from it).
For any $E\in\edges(\triang)$ and any $\square\in Q(\triang)$,
let $G(E,\square)\in\{0,1,2\}$ be
the number of tetrahedral edges facing the quadrilateral~$\square$ in the edge class of $E$.

\begin{definition}
\label{circlevaluedanglestructure}
    An \emph{\Sval\ angle structure} on a triangulation $\triang$ is a function
    $\omega : Q(\triang)\to S^1$, where $S^1 \subset \CC$ is the unit circle, 
    satisfying
    \begin{align}
        \omega(\square)\omega(\square')\omega(\square'') \label{SAS:tetrah_prod=-1}
            &= -1\ \text{for every}\ \square\in Q(\triang),\ \text{and}\\
        \prod_{\square\in Q(\triang)} \omega(\square)^{G(E,\square)} \label{SAS:edge_prod=1}
            &= 1\ \text{for every}\ E\in\edges(\triang).
    \end{align}
    The set of all \Sval\ angle structures on $\triang$ is denoted by
    $\SAS(\triang)$.
\end{definition}

We will also discuss a refinement of this concept involving extra constraints coming from the boundary. In what follows, the term \emph{peripheral curve} stands for
an oriented, homotopically non-trivial simple closed curve in $\del M$.
By placing a peripheral curve~$\gamma$ in normal position with respect to
the triangulation of $\partial M$ induced by $\triang$, 
we may
associate to $\gamma$ the completeness equation coefficients $G(\gamma, \square)\in\ZZ$, $\square\in Q(\triang)$
in a standard way \cite{thurston-notes,neumann-zagier}. To be precise, $G(\gamma, \square)$ is the sum of a $+1$ for every time $\gamma$ subtends a corner of a triangle in the induced triangulation of the boundary corresponding to the quad-type $\square$ in an anti-clockwise fashion, and a $-1$ for every time it subtends such a corner in a clockwise fashion. (See page 262 of \cite{neumann1990}.)

Then the \emph{multiplicative angle-holonomy of $\omega$ around $\gamma$} is defined by the expression
    \begin{equation}
    \ahol{\omega}(\gamma) =
        \prod_{\square\in Q(\triang)} {\omega(\square)}^{G(\gamma,\square)}
    \in S^1.
\end{equation}

\begin{definition}
\label{peripherallytrivialcirclevaluedanglestructure}
    An \Sval\ angle structure $\omega$ on the triangulation $\triang$ is said to be \emph{peripherally trivial} if for every peripheral curve $\gamma$, the multiplicative angle-holonomy of $\omega$ around $\gamma$
equals 1. The set of all peripherally trivial \Sval\ angle structures on $\triang$ is denoted by $\SAS_0(\triang)$.
\end{definition}

Of importance to this work are
two settings where \Sval\ angle structures arise naturally. 
The first setting is that of Thurston's gluing and completeness equations, which are reviewed in Section~\ref{sec:gluing-equations}. If $z: Q(\triang)\rightarrow\CC\setminus\{0,1\}$ is an algebraic solution of Thurston's gluing equations, then a corresponding \Sval\ angle structure is obtained as the ``phase part'' of the complex shape function $z$: 
\begin{equation}\label{corresponding-angle-structure}
\omega(\square)=\frac{z(\square)}{|z(\square)|}\in S^1.
\end{equation}
If, in addition, $z$ satisfies the completeness equations, then $\omega\in\SAS_0(\triang)$. 

The second setting where \Sval\ angle structures arise naturally 
is the meromorphic 3D-index of Garoufalidis and Kashaev \cite{garoufalidis-kashaev}. Their construction associates to an ideal triangulation a meromorphic function of $2k$ variables, where $k$ denotes the number of boundary components,
leading to the definition of
a topological invariant of the underlying 3-manifold. In the case that the ideal triangulation admits a strict angle structure, the meromorphic 3D-index can be expressed in terms of the combinatorially defined $q$-series 3D-index of Dimofte, Gaiotto and Gukov \cite{dimofte-gaiotto-gukov 3D index,dimofte-gaiotto-gukov gauge theories,stavros}.

The meromorphic 3D-index of an ideal triangulation $\triang$ is defined as a certain \emph{state integral}. For every edge~$E\in\edges(\triang)$, there is one integration variable~$y_E$
taking values in the unit circle~$S^1\subset\CC$.
The integrand is a product of \emph{tetrahedral weights}, one for each tetrahedron.
In the work \cite{hodgson-kricker-siejakowski} the authors show,
by performing a change of variables on these $S^1$-valued edge variables, that the evaluation at unity of the Garoufalidis-Kashaev integral can be expressed as an integral over exactly one of the connected components of the manifold $\SAS_0(\triang)$ of peripherally trivial \Sval\ angle structures.
This observation provides the motivation for our study of the properties of the space
of \Sval\ angle structures in this paper.

\subsection{General properties of the manifolds of circle-valued angle structures}

Both of the sets $\SAS(\triang)$ and $\SAS_0(\triang)$ acquire real manifold structures as subsets of $\mathbb{C}^{3N}$. These sets are straightforward to compute. 
Note that in each case they are the set of solutions in $(S^1)^{3N}\subset\mathbb{C}^{3N}$ to a system of monomial equations of the form
\[
\prod_{\square\in Q(\triang)}(\omega(\square))^{n_i(\square)}=\varepsilon_i. 
\]
Here $i$ indexes the equation, $n_i(\square)\in \mathbb{Z}$ for all $\square\in Q(\triang)$, and $\varepsilon_i=+1$ or $-1$.

By using elementary row and column operations to reduce the matrix of exponents to Smith Normal Form, one can diagonalize this system to one of the form
\[
(\omega'_1)^{m_1}=\varepsilon'_1,\ \ldots\ ,\ (\omega'_{d})^{m_{d}}=\varepsilon'_{d},
\]
where the $m_i\in\mathbb{Z}$ and $\varepsilon'_i$ equals $1$ or $-1$.

This confirms that the sets $\SAS(\triang)$ and $\SAS_0(\triang)$ are indeed real manifolds, and in fact are disjoint collections of tori. These facts were first observed by Luo \cite{luo-volume}. It is also easy to read off from this form of the equations the number of connected components. (It is the product of the non-zero $|m_j|$.)

Our investigation into these structures began with computations in a large number of examples of the number of connected components of the spaces \cite{chen}. These computations led us to the following theorem.

\begin{theorem}\label{intro-connected-component-counts}
Consider an ideal triangulation~$\triang$, consisting of $N$ tetrahedra, of a connected
orientable 3-manifold $M$ with $k$ toroidal ends.
\begin{itemize}
    \item
    The space $\SAS(\triang)$ consists of
    $\left|\Ker\left(\iota^*:H^2(M;\FF_2)\rightarrow H^2(\partial M;\FF_2)\right)\right|$ tori of real
    dimension $N+k$, where $\iota$ denotes the inclusion map of the boundary $\partial M$ into $M$.
    \item
    The space $\SAS_0(\triang)$ consists of $|H^2(M,\partial M;\FF_2)|$ tori of real dimension $N-k$.
\end{itemize}
\end{theorem}
The dimensions of the above spaces are the same as those of the corresponding familiar spaces of
real angle structures, and are already known for that reason, cf.~\cite{luo-volume,futer-gueritaud}. We will review these details in Section~\ref{angle_structures_section}. We also remark that, by using Poincar\'e duality, 
 the above component counts of $\SAS(\triang)$ and $\SAS_0(\triang)$ may alternatively be given as $|H_1(\hat{M};\FF_2)|$ and
$|H_1(M;\FF_2)|$, where $\hat{M}$ denotes the end compactification of $M$, which
results from filling in the vertices in the ideal triangulation~$\triang$.

Our first proofs of this theorem were combinatorial, obtained by interpreting these structures in terms of Walter Neumann's complex associated to an ideal triangulation \cite{neumann1990}. But it turned out there was a deeper, geometric story to uncover here, prompted by the observation that the aforementioned cohomology groups 
classify lifting obstructions for 
$\PSLC$-representations. 

\subsection{Circle-valued angle structures and obstruction theory}

In this section we will introduce our main results. We will use natural geometric ideas to concretely construct bijections between the set of components of the spaces $\SAS(\triang)$ and $\SAS_0(\triang)$ and the cohomology groups that count them, as described in Theorem~\ref{intro-connected-component-counts}.

The algebraic 
setting is the problem of lifting representations 
$\rho\colon\pi_1(M)\rightarrow \PSLC$ to representations $\tilde{\rho}\colon\pi_1(M)\rightarrow\SLC$.  We also consider the refined problem of lifting boundary-parabolic representations to $\PSLC$ to boundary-unipotent representations to $\SLC$.

Recall that the central extension 
\[
0\to\{\pm1\}\rightarrowtail\SLC\to\PSLC\to0
\]
is non-trivial, so a general representation $\rho: \pi_1(M)\to\PSLC$
may not admit a lift to $\SLC$. 
This failure to lift can be measured by an obstruction class. To each $\rho\colon \pi_1(M)\to\PSLC$ we can associate a cohomology class in
$H^2(M;\{\pm1\})\cong H^2(M;\FF_2)$ such that the representation $\rho$ lifts to $\SLC$ if and only if the class is zero.
See \cite{culler-lifting}, \cite[\S2]{calegari} or \cite[\S2]{acuna-montesinos} for the details of the construction of this class. 
For convenience, we briefly recall the idea in the next two paragraphs. 

Represent $M$ 
by a cell-complex. (For example, of particular importance in this work are the cell-complexes obtained from $\triang$ by taking \emph{doubly-truncated tetrahedra} and \emph{cylinders around edges}, 
as described in Section~\ref{truncated-triangulations}.) Choose a basepoint in the 0-skeleton and choose orientations of the edges. Attach elements $g(e)\in\PSLC$ to the edges $e$ so that when a loop $\gamma$ in the 1-skeleton is obtained as a sequence of edges, the product of the corresponding elements of $\PSLC$ (with power $+1$ when the edge is traversed following the chosen orientation, and $-1$ when it is traversed against the orientation) gives the element $\rho([\gamma])\in \PSLC$. It is easy to see, for example 
by using a spanning-tree, that such a choice exists. 

Now we try to lift this representation. For each edge $e$, choose an element $\tilde{g}(e)\in\SLC$ lifting $g(e)$. This defines a 2-cocycle valued in the abelian group $\FF_2$: a 2-cell is mapped to $0$ if the product of the $\tilde{g}(e)$ around the boundary of the cell is $I\in \SLC$ and is mapped to $1\in \FF_2$ if the product is $-I$.
This cocycle is zero if and only if this assignment $\tilde{g}$ defines a representation to $\SLC$. If the cocycle is non-zero, we can attempt to fix it by adjusting our choices of lifts $\tilde{g}(e)$ of the $g(e)$. 
This is equivalent to performing a coboundary modification on $g$.
It follows that the representation $\rho$ lifts if and only if the 2-cocycle thus constructed is trivial in the cohomology group $H^2(M;\FF_2)$. 

There is also a relative version of the lifting problem which
concerns the case when $\rho$ is \emph{boundary-parabolic};
in terms of M\"obius transformations, this means
that $\rho$ takes the fundamental group of each component of $\partial M$
to a group conjugate into $\{z\mapsto z+c: c\in\CC\}<\PSLC$.
In this case, we may wish to lift $\rho$ to a \emph{boundary-unipotent}
\nword{\SLC}{representation}, i.e., a representation mapping the boundary
subgroups of $\pi_1(M)$ into subgroups conjugate to
$\displaystyle\left\{\begin{bmatrix}1&b\\0& 1\end{bmatrix}: b\in\CC\right\}<\SLC$.
The obstruction classes to lifting boundary-parabolic \nword{\PSLC}{representations}
to boundary-unipotent \nword{\SLC}{representations} are relative cohomology classes
in $H^2(M,\partial M;\FF_2)$. See~\cite[\S9.1]{garoufalidis-thurston-zickert} for the details of this version.

\medskip
Now we present our main theorem. In what follows, we treat all cohomology groups with \nword{\FF_2}{coefficients} as groups with the discrete topology.

\begin{theorem}\label{obstruction-theorem}
There exist maps
\[
    \Phi:\SAS(\triang)\to H^2(M;\FF_2)
    \quad\text{and}\quad
    \Phi_0:\SAS_0(\triang)\to H^2(M,\partial M;\FF_2)
\]
satisfying the following conditions:
\begin{enumerate}[(i)]
\item\label{item:continuity}
$\Phi$ and $\Phi_0$ are continuous, i.e., constant on the
connected components of their domains.
\item\label{obstructions-bijectivity}
The induced map on the connected components,
\[
    \pi_0(\Phi_0):\pi_0\bigl(\SAS_0(\triang)\bigr)\to H^2(M,\partial M;\FF_2),
\]
is a bijection. Moreover, the induced map
\[
    \pi_0(\Phi):\pi_0\bigl(\SAS(\triang)\bigr)\to H^2(M;\FF_2)
\]
is an injection whose image coincides with the kernel of the map
\begin{equation}\label{restriction-to-boundary-H2F2}
    \iota^* : H^2(M;\FF_2) \to H^2(\partial M;\FF_2)
\end{equation}
induced by the inclusion of the boundary into $M$.
\item\label{item:correct-obstructions}
If $z$ is an algebraic solution of Thurston's gluing equations on $\triang$
defining a representation $\rho:\pi_1(M)\to\PSLC$
and $\omega(\square)=z(\square)/|z(\square)|$
is the associated circle-valued angle structure, then $\Phi(\omega)\in H^2(M;\FF_2)$ is the
cohomological obstruction to lifting $\rho$ to an $\SLC$ representation.
Moreover, if $\rho$ is boundary-parabolic,
then $\Phi_0(\omega)\in H^2(M,\partial M;\FF_2)$ is the cohomological obstruction
to lifting $\rho$ to a boundary-unipotent $\SLC$ representation.
\end{enumerate}
\end{theorem}

In the last part of the theorem, recall that an algebraic solution of Thurston's gluing equations $z\colon Q(\triang)\rightarrow\CC\setminus\{0,1\}$ determines a corresponding \Sval\ angle structure, see equation~(\ref{corresponding-angle-structure}). 
Further, $z$ also determines a pseudo-developing map $\tilde{M}\rightarrow\mathbb{H}^3$
 where $\tilde{M}$ is the universal cover of the interior of $M$, 
 and, through this, a conjugacy class of representations $\pi_1(M)\rightarrow \PSLC$. (See \cite{yoshida-ideal} and Section~\ref{sec:gluing-equations}.) If, in addition, the solution satisfies the completeness equations, then the constructed representation is boundary-parabolic.

It is a corollary of part~\eqref{obstructions-bijectivity} of this theorem that
the numbers of connected components of $\SAS(\triang)$ and of $\SAS_0(\triang)$
are both finite, independent of $\triang$, and given by the cardinality of the cohomology groups with $\FF_2$-coefficients stated in Theorem~\ref{intro-connected-component-counts}. In particular, both numbers are always powers of $2$.    

This theorem relates the spaces $\SAS(\triang)$ and $\SAS_0(\triang)$
of circle-valued angle structures on an ideal triangulation~$\triang$
to the lifting problem for $\PSLC$ representations.
This relationship mirrors the properties of the `Ptolemy varieties' studied
by Zickert and others \cite{zickert-volrep,garoufalidis-thurston-zickert}.
But it is worth observing that in that theory, the obstruction theory for the lifting problem is used as a starting point for the various definitions of spaces, whereas
here, the spaces of \Sval\ angle structures are defined without reference
to the obstruction classes, and their close relationship with the obstruction theory is an emergent property.

Part \eqref{item:correct-obstructions} of the theorem can be summarised by the following commutative diagram:
\[
     \begin{tikzcd}[row sep=15mm]
         {
             \left\{
                \begin{array}{c}
                     \text{Algebraic solutions}\\
                     z:Q(\triang)\to\CC \\
                     \text{of Thurston's equations}
                 \end{array}
             \right\}
         }
         \arrow{r}
         \arrow[swap]{d}{\displaystyle\omega(\square)=\frac{z(\square)}{|z(\square)|}}
         &
         {
             \left\{
                 \begin{array}{c}
                     \text{Conjugacy classes of} \\
                     \text{boundary-parabolic}\\
                     \PSLC\text{-representations}
                 \end{array}
             \right\}
         }
         \arrow{d}{\displaystyle\Obs_0}
         \\
         {
             \SAS_0(\triang)
         }
         \arrow{r}{\displaystyle\Phi_0}
         &
         {
             H^2(M, \partial M; \FF_2),
         }
     \end{tikzcd}
 \]
where $\Obs_0$ is the map which takes any boundary-parabolic $\PSLC$-representation $\rho$ to the obstruction class to lifting $\rho$ 
to a boundary-unipotent $\SLC$-representation.

\subsection{Application to the Garoufalidis-Kashaev meromorphic 3D-index}

We shall finish this introduction with a few remarks about how this theorem is used in the work \cite{hodgson-kricker-siejakowski} to study the $q=e^{-t}, t\rightarrow 0^+$ asymptotics of the meromorphic 3D-index of Garoufalidis and Kashaev \cite{garoufalidis-kashaev} in the case of a 1-cusped manifold with an ideal triangulation admitting a geometric solution to Thurston's gluing and completeness equations.

Our investigation began by doing extensive numerical computations of the meromorphic 3D-index. We discovered that the asymptotics seemed to take the form of a sum of contributions indexed by a certain collection of boundary-parabolic representations of $\pi_1(M)$, including both complex and real representations. One puzzle was, exactly which representations were appearing? 

The main theorem of this paper, Theorem~\ref{obstruction-theorem}, answered this puzzle as follows. As mentioned above, the meromorphic 3D-index can be re-expressed as an integral over exactly one of the components of $\SAS_0(\triang)$: in fact, the component containing the \Sval\ structure coming from the complete hyperbolic representation.

A stationary phase analysis of the integral leads to the conclusion that the contributions to the asymptotics come from the critical points of the volume functional on this component. At each such critical point, the optimisation theory of Casson-Rivin, together with some conjectures of the current authors concerning the real structures, demonstrates the existence of an algebraic solution of Thurston's equations whose corresponding \Sval\ angle structure is that critical point \cite{futer-gueritaud, hodgson-kricker-siejakowski}. 

Theorem~\ref{obstruction-theorem} lets us characterise those contributing representations topologically. They are precisely the boundary-parabolic representations whose corresponding obstruction class is the same as the obstruction class coming from the complete geometric structure.

\par\medskip\noindent
\textbf{Acknowledgements.}
The computations of the connected components of the spaces of circle-valued angle structures were performed by the student Tracey Chen during a summer research project at the University of Melbourne in January-February 2020 under the supervision of Craig Hodgson. We thank her for her contribution.
This research has been partially supported by
grants DP160104502 and DP190102363 from the Australian Research Council.  The
project also received support through the AcRF Tier~1 grants RG~32/17 and RG~17/23 from the
Singapore Ministry of Education.  Rafa\l\ Siejakowski was supported by grant
\char`#2018/12483-0 of the S\~ao Paulo Research Foundation (FAPESP).  Our work
has benefited from the support of Nanyang Technological University and the
University of Melbourne.  We would like to thank these institutions for their
hospitality.

%==============================================================================
% Section 2: Preliminaries
%==============================================================================
\section{Preliminaries on ideal triangulations and angle structures}
\label{sec:preliminaries}

In this paper we use the notation $\{\pm 1\}$ to denote 
the multiplicative group with 2 elements. 
We use $\mathbb{Z}_2$ to denote the field with two elements, whose \emph{additive} group is isomorphic to $\{\pm 1\}$ in the obvious way.
This allows us to treat the respective cohomology groups as vector spaces.

\subsection{Thurston's gluing equations}\label{sec:gluing-equations}

Let $M$ be a compact, connected, orientable \nword{3}{manifold} with boundary consisting of $k \ge 1$ tori, and let $\triang$ be an ideal triangulation of the interior of $M$. 
In the Introduction we defined
 the incidence numbers $G(E,\square)$, where $E\in \edges(\triang)$ and $\square\in Q(\triang)$. Following Neumann and Zagier \cite{neumann-zagier}, it is common to organise them
into matrices.
To this end, we fix a numbering of both the tetrahedra and the edges with integers $\{1,\dotsc,N\}$,
so that $\edges(\triang)=\{E_1,\dotsc,E_N\}$.
For every $j\in\{1,\dotsc,N\}$, we choose a distinguished normal
quadrilateral~$\square_j\subset\tet_j$ in the $j$th tetrahedron.
With these choices, the \emph{gluing matrix} of $\triang$ is defined to be the matrix
\begin{align}\label{def-gluing-matrix}
    \mathbf{G} &= [\,G\;|\;G'\;|\;G''\,] \in\Mat_{N\times 3N}(\ZZ),\ \text{where for any}\
    m,n\in\{1,\dotsc,N\}\ \text{we have}\\
    G_{m,n} &= G(E_m,\square_n),\quad
    G'_{m,n} = G(E_m,\square'_n),\quad
    G''_{m,n} = G(E_m,\square''_n).\notag
\end{align}
Here, and throughout the paper, $\Mat_{m\times n}(\mathcal{R})$
denotes the set of $m\times n$ matrices with entries in a ring $\mathcal{R}$.

Recall also that in the Introduction we discussed that we can assign incidence numbers to peripheral curves. 
By numbering the connected components of $\del M=T_1\sqcup\dotsb\sqcup T_k$,
and fixing, for every $1\leq j\leq k$, a pair of peripheral curves~$\mu_j, \lambda_j\subset T_j$
generating $H_1(T_j,\ZZ)$, we obtain the \emph{peripheral gluing matrix}~$\mathbf{G_\del}$,
\begin{align}\label{peripheral-gluing-matrix}
    \mathbf{G_\del} &= [\,G_\del\;|\;G'_\del\;|\;G''_\del\,] \in\Mat_{2k\times 3N}(\ZZ),\ \text{where}\\
        (G_\del)_{2j-1,n} &= G(\mu_j,\square_n),\quad\notag
        (G_\del)_{2j,n} = G(\lambda_j,\square_n),\quad 1\leq j \leq k;
\end{align}
$G'_\del$, $G''_\del$ are defined analogously.

With the above notations, Thurston's hyperbolic gluing equations,
as presented by Neumann and Zagier~\cite{neumann-zagier}, have the form
\begin{equation}\label{edge-consistency-equations}
    \prod_{\square\in Q(\triang)} {z(\square)}^{G(E_m,\square)}
        =\prod_{n=1}^N {z_n}^{G_{m,n}}{z'_n}^{G'_{m,n}}{z''_n}^{G''_{m,n}}
                = 1\ \text{for all}\ m,
\end{equation}
where the shape parameter variables $z \in\CC\setminus\{0,1\}$ satisfy $z_n = z(\square_n)$ and
\begin{equation}\label{z-relations}
    z(\square') = \frac{1}{1-z(\square)}\ \text{for all}\
                  \square\in Q(\triang),
\end{equation}
i.e., $z'_n = 1/(1-z_n)$ and $z''_n = (z_n-1)/z_n$ for all $n$.
After taking logarithms, we may rewrite equation~\eqref{edge-consistency-equations}
in matrix form as
\begin{equation}\label{edge-consistency-general}
    GZ + G'Z' + G''Z'' \in 2\pi i \ZZ^{\edges(\triang)},
\end{equation}
where $Z=(\log z_1, \dotsc,\log z_N)$, $Z'=(\log z'_1,\dotsc,\log z'_N)$,
and $Z''=(\log z''_1,\dotsc,\log z''_N)$.
Of particular importance is the geometric case, where we impose
the equations
\begin{equation}\label{edge-consistency-logarithmic}
    GZ + G'Z' + G''Z'' = 2\pi i\, (1,1,\dotsc,1)^\transp,
\end{equation}
which can be used in the search for hyperbolic structures, cf.~\cite[eq.~4.2.2]{thurston-notes}.

An \emph{algebraic solution} $z=\exp(Z)$ is any solution of
\eqref{edge-consistency-equations} subject to \eqref{z-relations}
with $z(\square)\in\CC\setminus\{0,1\}$ for all $\square\in Q(\triang)$.
Yoshida~\cite{yoshida-ideal} showed
that any algebraic solution
of the gluing equations 
determines a conjugacy
class of a representation~$\rho\colon\pi_1(M)\to\PSLC$.
When the log-parameters~$Z$ additionally satisfy
\begin{equation}\label{completeness-equations-mod-2pi}
    G_\del Z + G'_\del Z' + G''_\del Z'' \in 2\pi i \ZZ^{2k},
\end{equation}
then $\rho$ is boundary-parabolic.

%--------------------------------------------------------------------------------------------------
\subsection{Angle structures on ideal triangulations}\label{angle_structures_section}

Our study of circle-valued angle structures will need a number of closely related versions of the concept of angle structure, which we'll discuss in this section. 

\begin{definition}
\label{angle structures}
    A \emph{generalised angle structure} on the ideal triangulation $\triang$
    is a function $\alpha : Q(\triang)\to\RR$ satisfying
    \begin{align}
        \alpha(\square)+\alpha(\square')+\alpha(\square'') \label{AS:tetrah_sum=pi}
            &= \pi\ \text{for every}\ \square\in Q(\triang),\ \text{and}\\
        \sum_{\square\in Q(\triang)} G(E,\square)\,\alpha(\square) \label{AS:edge_sum=2pi}
            &= 2\pi\ \text{for every}\ E\in\edges(\triang).
    \end{align}
    The set of all generalised angle structures on $\triang$ is denoted by $\Areal(\triang)$.
\end{definition}
The study of angle structures began with Casson and Rivin. They studied \emph{strict angle structures}, which are generalised angle structures with the extra condition that $\alpha(\square)>0$ for all $\square\in Q(\triang)$. Angle structures arise from studying the linear part of Thurston's equations; see~\cite{futer-gueritaud} for an exposition.%ajk

Since the columns of the matrices $\mathbf{G}$, $\mathbf{G_\del}$ correspond to normal quadrilaterals,
we may simply write \eqref{AS:edge_sum=2pi} as $\mathbf{G}\alpha = (2\pi,\dotsc,2\pi)^\transp$ with
the understanding that the normal quads are ordered in accordance with the ordering of tetrahedra, i.e.,
\begin{equation}\label{quad-ordering}
    Q(\triang) = \{ \square_1, \square_2, \dotsc, \square_N;\;
            \square'_1, \square'_2, \dotsc, \square'_N;\;
            \square''_1, \square''_2, \dotsc, \square''_N\}.
\end{equation}

Observe that the formula~$\omega(\square)=e^{i\alpha(\square)}$ defines
an ``exponential map''~$\exp : \Areal(\triang)\to\SAS(\triang)$.
As $\Areal(\triang)\subset\RR^{Q(\triang)}$ is a real affine space (and is therefore connected),
the image of the exponential map lies entirely in a single connected component of $\SAS(\triang)$,
which we shall call the \emph{geometric component}
and denote $\SASdistinguished(\triang)$.
More generally, all of $\SAS(\triang)$ can be viewed as the set of exponentiated
real-valued \emph{pseudo-angles}, i.e., vectors $\alpha : Q(\triang)\to\RR$ which solve the angle
equations~\eqref{AS:tetrah_sum=pi}, \eqref{AS:edge_sum=2pi} modulo $2\pi\ZZ$.

The maps in our main theorem, Theorem~\ref{obstruction-theorem}, will be constructed using the following important special case of an \Sval\ angle structure.  
\begin{definition}
\label{plusorminusone structures}
    A \emph{$\{\pm 1\}$-valued angle structure} on the ideal triangulation $\triang$ is a \Sval\ angle structure $\omega : Q(\triang)\to S^1$ such that  $\omega(\square)\in \{\pm 1\}$ for all $\square\in Q(\triang)$.
\end{definition}

When $\alpha$ is a generalised angle structure
and $\gamma$ is a peripheral curve, the \emph{angle-holonomy} of
$\alpha$ along $\gamma$ is given by
\begin{equation}\label{def-angle-holonomy}
    \ahol{\alpha}(\gamma)= \sum_{\square\in Q(\triang)} G(\gamma, \square) \alpha(\square)\in\RR.
\end{equation}
We say that $\alpha$ has \emph{vanishing peripheral angle-holonomy},
or that $\alpha$ is \emph{peripherally trivial},
if $\ahol{\alpha}(\mu_j)=\ahol{\alpha}(\lambda_j)=0$ for all $j\in\{1,\dotsc,k\}$,
i.e., when
\(
    \mathbf{G_\del} \alpha = 0
\).
Similarly, we may define the \emph{multiplicative angle-holonomy} 
 of an \Sval\ angle structure
$\omega$ along $\gamma$ by the formula
\begin{equation}\label{peripheral_S1_holonomy_defn}
    \ahol{\omega}(\gamma) =
        \prod_{\square\in Q(\triang)} {\omega(\square)}^{G(\gamma,\square)}
    \in S^1.
\end{equation}

We shall sometimes need to consider spaces of angle structures with
prescribed peripheral angle-holonomies. 
For simplicity, suppose that $M$ has only one torus boundary component,
which we orient in accordance with the orientation of $M$, using the convention of
``outward-pointing normal vector in the last position''.
We may further assume that the peripheral curves~$\mu$, $\lambda$ have
intersection number~$\intersection(\mu,\lambda)=1$ with respect to this orientation.
For any $(x,y)\in\RR^2$, we define
\begin{equation}\label{def-Axy}
    \Areal_{(x,y)}(\triang) = \{\alpha\in\Areal(\triang)\ |\
        \bigl(\ahol{\alpha}(\mu), \ahol{\alpha}(\lambda)\bigr)=(x,y)\}.
\end{equation}
Likewise, for two elements $\xi,\eta\in S^1$ we shall consider
\Sval\ angle structures with prescribed multiplicative angle-holonomy,
\begin{equation}\label{def-SASxy}
    \SAS_{(\xi,\eta)}(\triang) = \{\omega\in\SAS(\triang)\ |\
    \bigl(\ahol{\omega}(\mu), \ahol{\omega}(\lambda)\bigr)=(\xi,\eta)\}.
\end{equation}
An important special case is presented by the spaces of \emph{peripherally trivial}
angle structures.
In order to lighten the notation, we set
\begin{equation*}
    \Areal_0(\triang) := \Areal_{(0,0)}(\triang),
    \qquad
    \SAS_0(\triang) := \SAS_{(1,1)}(\triang).
\end{equation*}
Note that the space $\SAS_0(\triang)$ is independent of the choice of the curves~$\mu$ and~$\lambda$.

\begin{definition}\label{def:Zvals}
\begin{enumerate}[(i)]
\item
An \emph{integer-valued angle structure} on the triangulation $\triang$ is a
vector $f\in\ZZ^{Q(\triang)}$ such that
$\pi f \in \Areal(\triang)$, where $\pi f$ denotes the function which sends $\square$ to $\pi$ times $f(\square)$.
\item
An integer-valued angle structure~$f$ is \emph{peripherally trivial} if $\pi f\in\Areal_0(\triang)$.
\end{enumerate}
\end{definition}
By a result of Neumann~\cite[Theorem~2]{neumann1990}, an ideal triangulation
of a manifold with toroidal boundary always admits a peripherally trivial integer-valued angle structure.
As a consequence, the set $\SAS_0(\triang)$ is always non-empty.
We provide a generalisation of Neumann's theorem in Theorem~\ref{generalization-of-6.1} below.
We refer to \cite{luo-volume,garoufalidis-hodgson-hoffman-rubinstein} for more details on angle
structures of various types.

%--------------------------------------------------------------------------------------------------
\subsection{Tangential angle structures}\label{subsec:TAS}
\begin{definition}
The space of \emph{tangential angle structures}~$\TAS(\triang)\subset\RR^{Q(\triang)}$ is the vector
space consisting of vectors~$w$ satisfying the equations
\begin{align}
    w(\square) + w(\square') + w(\square'')&=0,\quad \square\in Q(\triang),\label{TAS:tet_sum=0}\\
    \label{TAS:edge_sum=0}
    \mathbf{G}\,w&=0.
\end{align}
Moreover, define $\TAS_0(\triang)\subset\TAS(\triang)$ to be the subspace cut out by the equations
\(
    \mathbf{G_\del}\,w = 0
\).
\end{definition}
Since the above equations are homogeneous versions of the angle structure
equations \eqref{AS:tetrah_sum=pi}--\eqref{AS:edge_sum=2pi},
$\TAS(\triang)$ is canonically identified with the tangent space to
the affine subspace $\Areal(\triang)\subset\RR^{Q(\triang)}$ at every one of its points,
and by using the exponential map, it is also identified
with the tangent space to $\SAS(\triang)$.
By the same reasoning, $\TAS_0(\triang)$ is the tangent space to the subspaces
$\Areal_{(x,\,y)}(\triang)$ and $\SAS_{(\xi,\,\eta)}(\triang)$ with fixed
peripheral angle-holonomy; see~\cite[p.~307]{luo-volume} for more details.

Based on the treatment of Futer and Gu\'eritaud~\cite[\S4]{futer-gueritaud},
we introduce explicit bases for $\TAS(\triang)$ and $\TAS_0(\triang)$.
For every $j\in\{1,\dotsc,N\}$, we define the \emph{leading--trailing deformation about
the edge}~$E_j\in\edges(\triang)$
to be the vector $l_j\in\ZZ^{Q(\triang)}$ given by the formula
\begin{equation}\label{definition-leading-trailing}
    l_j(\square) = G(E_j,\square'') - G(E_j,\square').
\end{equation}
The vectors $\{l_1,\dotsc,l_N\}$ span $\TAS_0(\triang)$,
cf.~\cite[Proposition~4.6]{futer-gueritaud}.
Using the matrix notation of \eqref{def-gluing-matrix}, we may define
\begin{equation}\label{definition-L}
    \begin{aligned}
        \mathbf{L} &= [\,L\;|\;L'\;|\;L''\,] \in\Mat_{N\times 3N}(\ZZ),\ \text{where}\\
        L &= G''-G',\quad L' = G-G'', \quad L'' = G'-G.
    \end{aligned}
\end{equation}
With this notation, the vector $l_j$ can be understood as the \nword{j}{th} row of $\mathbf{L}$.
Note that we always have $-l_N = l_1+\dotsb+l_{N-1}$.
In the case of a single boundary component,
this is the only nontrivial linear relation
among the vectors $l_j$, so that
after removing the last vector ($j=N$), the set~$\{l_1,\dotsc,l_{N-1}\}$ becomes
a basis of $\TAS_0(\triang)$, cf. \cite[\S4]{futer-gueritaud}.
In this case, it is convenient to define
\begin{equation}\label{definition-L_star}
    \begin{aligned}
    \mathbf{L_*} &= [\,L_*\;|\;L'_*\;|\;L''_*\,] \in\Mat_{(N-1)\times 3N}(\ZZ),\\
    [\mathbf{L_*}]_{j,\square} &= [\mathbf{L}]_{j,\square},\quad 1\leq j\leq N-1,
    \end{aligned}
\end{equation}
so that the rows of $\mathbf{L_*}$ form an explicit ordered basis of $\TAS_0(\triang)$.

More generally, Futer and Gu\'eritaud~\cite{futer-gueritaud} consider
leading-trailing deformations along oriented peripheral curves in general position with
respect to $\triang$.
When $\gamma$ is such a peripheral curve, we may define
\begin{equation}\label{peripheral-LTD}
    l_\gamma(\square) = G(\gamma,\square'') - G(\gamma,\square').
\end{equation}
It suffices for our purposes to consider fixed meridian-longitude
pairs~$(\mu_j,\lambda_j)$, $1\leq j \leq k$ and the associated
leading-trailing deformation vectors $\{l_{\mu_j}, l_{\lambda_j}\}_{j=1}^k$.
These vectors arise as the rows of the matrix
\begin{equation}\label{def-L-peripheral}
    \begin{aligned}
    \mathbf{L}_\del &= [\,L_\del\;|\;L'_\del\;|\;L''_\del\,]\in\Mat_{2k\times 3N}(\ZZ),\\
    L_\del &= G''_\del - G'_\del,\quad
    L'_\del = G_\del - G''_\del,\quad
    L''_\del = G'_\del - G_\del.
    \end{aligned}
\end{equation}
In this notation, we have $\TAS_0(\triang)=\mathbf{L}\mkern-2mu^\transp(\RR^N)=\mathbf{L}_*^\transp(\RR^{N-1})$ and
$\TAS(\triang) = \TAS_0(\triang) + \mathbf{L}^\transp_\del(\RR^{2k})$; cf.~\cite[Proposition~4.6]{futer-gueritaud}.
Moreover, the dimension of $\TAS_0(\triang)$ is always $N-k$,
whereas the dimension of $\TAS(\triang)$ is $N+k$.
We refer to \cite[Proposition~3.2]{futer-gueritaud} and \cite[Proposition~2.6]{luo-volume} for more
details and proofs.

\begin{remark}\label{TAS-and-connected-components}
As a consequence of the properties of leading-trailing deformations discussed above,
two \Sval\ angle structures~$\omega_1$ and $\omega_2$ lie in the same connected component of
$\SAS_0(\triang)$ if and only if there exists a
vector~$r\in\Imag\mathbf{L}^\transp=\mathbf{L}^\transp(\RR^N)$ such that
\begin{equation}\label{TAS-connected-SAS}
    \omega_2(\square) = \omega_1(\square) e^{ir(\square)}\ \text{for all}\ \square\in Q(\triang).
\end{equation}
Likewise, $\omega_1$ and $\omega_2$ belong to the same connected component of $\SAS(\triang)$
if and only if a vector~$r$ satisfying \eqref{TAS-connected-SAS}
can be found in the larger space~$\Imag\mathbf{L}^\transp+\Imag\mathbf{L}_\del^\transp$.
\end{remark}

%--------------------------------------------------------------------------------------------------

%==============================================================================
% Section 3: Obstructions
%==============================================================================
\section{\texorpdfstring{Spaces of  $S^1$-valued angle structures: obstruction classes and components}{Spaces of circle-valued angle structures: obstruction classes and components}}
\label{sec:obstructions}
In this section we turn to our main theorem, Theorem~\ref{obstruction-theorem}. 
The presentation of the proof is organised as follows.

The foundation of our proof is Neumann's work on the combinatorics of ideal triangulations. We begin by reviewing in Section~\ref{neumann-review} some key definitions and results from \cite{neumann1990}.  

Section~\ref{pseudoangle-existence} is devoted to stating and proving Theorem~\ref{generalization-of-6.1}, which is a key technical preliminary in our theory. This theorem is a generalisation of a theorem from \cite{neumann1990} which shows the existence on any ideal triangulation $\triang$ of integer-valued angle structures satisfying a certain parity condition. Here we generalize this  statement to prove the existence of integer-valued \emph{pseudo-angle} structures satisfying the parity condition. 
These generalise angle-structures by allowing the right-hand sides of the gluing and completeness equations to be prescribed more generally.

In Section~\ref{truncated-triangulations} we turn to the main task of the proof, which is to build maps from \Sval\ angle structures to cohomology classes. We start by introducing the cell complexes called \emph{doubly-truncated ideal triangulations} that we will be working with. Then we show how to use them to define 
the \emph{rectangle map}. This map, denoted `$\rect$', takes a function in $\FF_2^{Q(\triang)}$ satisfying the mod-2 gluing equations to a 2-cocycle on the doubly-truncated ideal triangulation. 

In Section~\ref{obstruction-map-defn}, we use the map `$\rect$' to define the key maps $\Phi_0$ and $\Phi$ which map an \Sval\ angle structure to a cohomology class. 

In Section~\ref{proof-of-obstruction-theorem} we prove parts (i) and (ii) of Theorem~\ref{obstruction-theorem}. These are the parts of the theorem concerning the bijectivity of the maps $\pi_0(\Phi)$ and $\pi_0(\Phi_0)$.

In Section~\ref{sec:cocycles} we prove part (iii) of Theorem~\ref{obstruction-theorem}, which is the compatibility statement for representations arising from solutions of Thurston's equations. To do this, we need to start with an algebraic solution to Thurston's equations, then from it explicitly compute the corresponding obstruction cocycle.
We use Neumann's concept of a strong flattening to build an explicit $\SLC$ labelling of the edges of the complex which projects to the corresponding $\PSLC$ representation and is in addition boundary-compatible, according to Definition~\ref{def:nice-cocycle}. 

\subsection{Combinatorics of 3-manifold triangulations following Neumann}\label{neumann-review}
In~\cite{neumann1990}, Neumann explored the relationship between the gluing equations
for an ideal triangulation and the homology and cohomology groups of the underlying manifold.
Here we will review the concepts and results from \cite[Sections~4--6]{neumann1990}
that we will need in order to prove Theorem~\ref{obstruction-theorem}.

\begin{definition}\label{definition-p}
Let $\mathcal{V}_\infty(\triang)$ be the set of ideal vertices of the triangulation $\triang$.
We define the map
\[
    p:\ZZ^{\edges(\triang)} \to \ZZ^{\mathcal{V}_\infty(\triang)},\qquad
    \bigl(p(x)\bigr)(V) = \sum_{E\in\edges(\triang)} \inc(E, V) \cdot x(E),
\]
where $\inc(E, V)\in\{0,1,2\}$ is the number of ends of $E$ incident to $V$.
\end{definition}

Let $\hat{\triang}$ be the simplicial complex obtained by filling in the ideal
vertices of $\triang$.
The mod~$2$ reduction of a vector in $\Ker p$ is easily seen to determine
a simplicial \nword{1}{cycle} in $\hat{\triang}$ with \nword{\FF_2}{coefficients}.
Neumann proved the following proposition as a part of Theorem~4.2 in \cite{neumann1990}.
\begin{proposition}[Neumann]\label{computing-homology-of-hat}
    Let $p$ be the map of Definition~\ref{definition-p}.
    The above identification of elements of $\Ker p$ with simplicial cycles
    induces an isomorphism $H_1(\hat{\triang};\FF_2) \cong \Ker p / \Imag_\ZZ\mathbf{L}$.\footnote{For $A\in\Mat_{m\times n}(\ZZ)$, we denote
by $\Imag_\ZZ A$ the lattice $A(\ZZ^n)\subset\RR^m$, and by $\Ker_\ZZ A$
the lattice $\ZZ^n\cap \Ker A\subset\RR^n$.}
\end{proposition}

Recall from \eqref{AS:tetrah_sum=pi}--\eqref{AS:edge_sum=2pi} that a generalised
angle structure~$\alpha$ is a real solution to the linear
system~$\mathbf{A}\alpha = \pi(\mathbf{1},\mathbf{2})^\transp$,
where $\mathbf{1} = (1,1,\dotsc,1)$, $\mathbf{2} = 2\cdot\mathbf{1}$ and
\begin{equation}\label{def-matrix-of-angle-eqs}
    \mathbf{A} =
    \begin{bmatrix}
        I & I  & I \\
        G & G' & G''
    \end{bmatrix},
\end{equation}
where $I$ is the $N\times N$ identity matrix.
The equations~\eqref{TAS:tet_sum=0}--\eqref{TAS:edge_sum=0} say
that every tangential angle structure~$\mathbf{w}\in\TAS(\triang)$ satisfies
$\mathbf{Aw}=0$, which implies $\mathbf{A}\mathbf{L}^\transp=0$ and consequently
$\Imag_\ZZ\mathbf{A}^\transp\subset\Ker_\ZZ\mathbf{L}$.
Let $\theta$ be an (oriented) peripheral curve in normal position with respect to 
the triangulation of $\partial M$ induced by $\triang$. 
As in Section~\ref{sec:preliminaries}, associated to $\theta$ is the
vector~$\bigl(G(\theta,\square)\bigr)_{\square\in Q(\triang)}\in\ZZ^{Q(\triang)}$
containing the coefficients of the completeness equation along $\theta$.
Now, the equality $\mathbf{G}_\partial\mathbf{L}^\transp=0$
implies, by duality, that
$\bigl(G(\theta,\square)\bigr)_{\square\in Q(\triang)}\in\Ker_\ZZ\mathbf{L}$.
Note that the gluing coefficients of any other curve homologous to $\theta$ will
differ from those of $\theta$ by an element of $\Imag_\ZZ\mathbf{A}^\transp\subset\ZZ^{Q(\triang)}$;
cf. also Remark~\ref{invariance-of-parity} below.
In this way, we may define the following homomorphism
\begin{align*}
    g: H_1(\partial M; \ZZ) &\to \Ker_\ZZ\mathbf{L}/\Imag_\ZZ\mathbf{A}^\transp\\
    g([\theta]) &= \bigl(G(\theta,\square)\bigr)_{\square\in Q(\triang)} + \Imag_\ZZ\mathbf{A}^\transp.
\end{align*}

In~\cite[\S5]{neumann1990}, Neumann also generalised the above construction to arbitrary curves in the
interior of $M$, but this generalisation only works over $\FF_2$.
Let $\theta$ be any closed unoriented path in $M$ in general position
with respect to $\triang$. 
When $\theta$ passes through a tetrahedron $\tet$, there is a well-defined
tetrahedral edge $e$ of $\tet$ common to the face where $\theta$ enters $\tet$ and
the face where $\theta$ leaves.
We may build a vector in $\FF_2^{Q(\triang)}$ by starting with $0$ and then,
every time $\theta$ crosses a tetrahedron, adding $1\in\FF_2$ to the
\nword{\square}{coordinate} of the vector, corresponding to the quad $\square$ facing $e$.
The sum of all such contributions along $\theta$ is the resulting vector
$\bigl(\square\mapsto G_2(\theta,\square)\bigr)\in\FF_2^{Q(\triang)}$.
In a similar way as before, we obtain a map
\begin{align*}
    g_2: H_1(M; \FF_2) &\to \Ker\mathbf{L}_2/\Imag\mathbf{A}_2^\transp\\
    g_2([\theta]) &= \bigl(G_2(\theta,\square)\bigr)_{\square\in Q(\triang)} + \Imag\mathbf{A}_2^\transp,
\end{align*}
where $\mathbf{L}_2\in\Mat_{N\times3N}(\FF_2)$ and $\mathbf{A}_2\in\Mat_{2N\times3N}(\FF_2)$ 
are the mod~$2$ reductions of the matrices $\mathbf{L}$ and $\mathbf{A}$, respectively.

In Section~5 of \cite{neumann1990}, Neumann proceeds to describe the relationship
between the cohomology group~$H^1(M;\FF_2)$ and the gluing equations.
To compute this cohomology group, we shall use the \nword{2}{complex}~$X$ dual
to the triangulation $\triang$, so that $H^1(M;\FF_2)\cong Z^1(X;\FF_2)/B^1(X;\FF_2)$.
For a face $F$ of the triangulation $\triang$, let $\chi_F\in C^1(X;\FF_2)$ be the cochain
taking the value~$1\in\FF_2$ on the \nword{1}{cell} dual to $F$ and vanishing everywhere else.
Given a tetrahedral edge $e$ in a tetrahedron $\tet$, we wish to assign to $e$ the sum
$\chi_{F_1}+\chi_{F_2}$ where $F_1$ and $F_2$ are the two faces of $\tet$ which meet at $e$
(see Figure~\ref{fig:theta}).
Note that if $e'$ is the edge of $\tet$ opposite to $e$, then the \nword{1}{cochain}
assigned to $e'$ is $\chi_{F_3}+\chi_{F_4}$, where $\{F_1,F_2,F_3,F_4\}$ is the complete
list of the four faces of $\tet$.
\begin{figure}
    \centering
    \begin{tikzpicture}
        \node[anchor=south west,inner sep=0] (image) at (0,0)
            {\includegraphics[width=0.3\columnwidth]{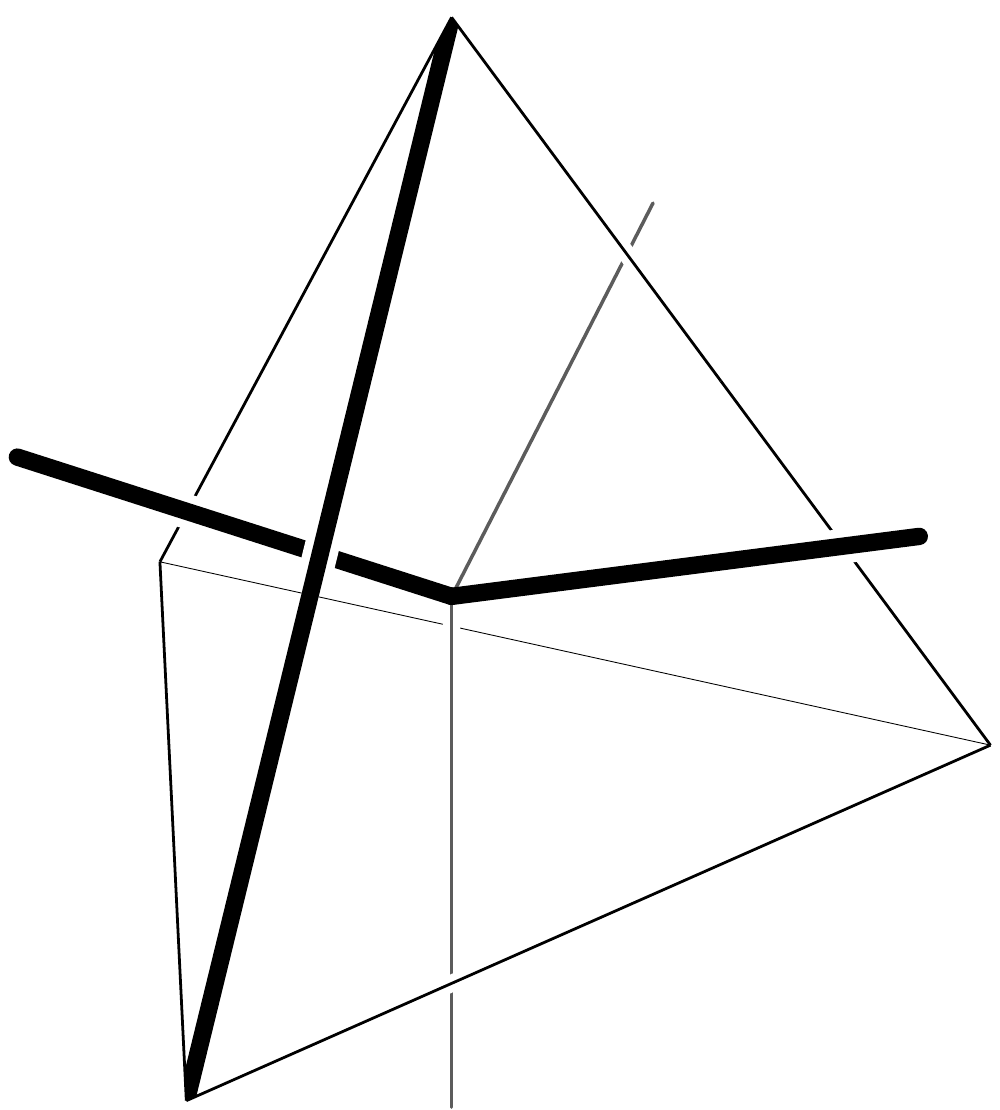}};
        \begin{scope}[x={(image.south east)},y={(image.north west)}]
            \node[anchor=north west] at (0.25, 0.3)  {$e$};
        \end{scope}
    \end{tikzpicture}
    \caption{To a tetrahedral edge~$e$ (bolded) we associate an \nword{\FF_2}{cochain}
    $\chi_{F_1}+\chi_{F_2}$    taking the value $1\in\FF_2$ on the bolded dual \nword{1}{cells} and
    the value $0\in\FF_2$ on the remaining \nword{1}{cells}.}\label{fig:theta}
\end{figure}
Since $\chi_{F_1}+\chi_{F_2}$ differs from $\chi_{F_3}+\chi_{F_4}$ by a coboundary (the image
under $\delta^0$ of the \nword{0}{cell} dual to $\tet$), this construction gives a well-defined map
\[
    \gamma'_2\iota: \ZZ^{Q(\triang)} \to C^1(X;\FF_2)/B^1(X;\FF_2),
\]
in Neumann's original notation.
In order to turn $\gamma'_2\iota$ into a map taking values in $H^1(M;\FF_2)$, we must impose cocycle
conditions which come from the \nword{2}{cells} of $X$.
Note that for an element $\mathbf{w}\in\ZZ^{Q(\triang)}$, the contribution of
$\delta^1\bigl(\gamma'_2\iota(\mathbf{w})\bigr)$ to the \nword{2}{cell} of $X$ dual to an edge
$E\in\edges(\triang)$ is
$\sum_{\square\in Q(\triang)} \left[l_E(\square) \mathbf{w}(\square)\right]_2$,
where $l_E$ is the leading-trailing deformation about $E$; cf.
also eq.~\eqref{Kronecker-pairing-gamma'2} below.
Therefore, for all $\mathbf{w}\in\Ker_\ZZ\mathbf{L}$,
$\delta^1\bigl(\gamma'_2\iota(\mathbf{w})\bigr)=0$.
Moreover, it is evident that for any $\mathbf{w}\in\Imag_\ZZ\mathbf{A}^\transp$
we have $\gamma'_2\iota(\mathbf{w})=0$.
In this way, $\gamma'_2\iota$ descends to a well-defined map
\begin{equation}
    \gamma'_2\iota: \Ker_\ZZ\mathbf{L}/\Imag_\ZZ\mathbf{A}^\transp \to H^1(M;\FF_2).
\end{equation}
The above construction can also be fully carried out over $\FF_2$, yielding an analogous map
\begin{equation}\label{gamma'_2}
    \gamma'_2: \Ker\mathbf{L}_2/\Imag\mathbf{A}_2^\transp \to H^1(M;\FF_2).
\end{equation}

\begin{proposition}[Neumann]\label{neumanns-exact-sequences}
The following sequences are exact:
\begin{align}\label{Neumanns-exact-mod2}
0\to H_1(M;\FF_2) \xrightarrow{g_2} &\Ker\mathbf{L}_2/\Imag\mathbf{A}_2^\transp
\xrightarrow{\gamma'_2} H^1(M;\FF_2) \to 0,\quad\text{and}\\
0\to H_1(\partial M;\ZZ) \xrightarrow{g} & \Ker\mathbf{L}/\Imag\mathbf{A}^\transp
\xrightarrow{\gamma'_2\iota} H^1(M;\FF_2) \to 0.\label{Neumanns-exact-integer}
\end{align}
\end{proposition}

In \cite{neumann1990}, the sequence~\eqref{Neumanns-exact-mod2} is given as the middle
row of the commutative diagram on p.~267, in Step~3 of the proof of Theorem~5.1,
whereas \eqref{Neumanns-exact-integer} is the top row of the diagram in Step~4 of the same proof. 

It is easy to evaluate a cohomology class~$\gamma'_2(\mathbf{w})$
on first homology classes with $\FF_2$ coefficients.
Combining equations (5.2) and (6.2) of \cite{neumann1990},
we see that for any vector $\mathbf{w}\in\Ker\mathbf{L}_2$
and for any unoriented closed curve $\theta$ in normal position with respect to $\triang$,
we have
\begin{equation}\label{Kronecker-pairing-gamma'2}
    \bigl({\gamma'_2}(\mathbf{w})\bigr)([\theta]) = \mathbf{w}\cdot [l_\theta]_2
    = \sum_{\square\in Q(\triang)} \mathbf{w}(\square)[l_\theta]_2(\square),
\end{equation}
where
\begin{equation}\label{curve-ltd-mod-2}
    [l_\theta]_2(\square) = G_2(\theta,\square'') - G_2(\theta,\square')
    \ \text{for all}\ \square\in Q(\triang).
\end{equation}

In \cite{neumann1990}, Neumann also introduces the concept of \emph{parity} along curves, which plays a central role in his analysis.

\begin{definition}
Let $\mathbf{x}\in\FF_2^{Q(\triang)}$ be an arbitrary \nword{\FF_2}{vector} and let $\theta$
be an unoriented curve in normal position with respect to $\triang$ and without
backtracking.
Then the \emph{parity of $\mathbf{x}$ along $\theta$}
is defined by the formula
\[
    \parity_\theta(\mathbf{x})
    = \sum_{\square\in Q(\triang)} G_2(\theta,\square)\mathbf{x}(\square)\in\FF_2.
\]
Moreover, we define the parity of a vector~$\mathbf{x}\in\ZZ^{Q(\triang)}$
as the parity of its modulo~$2$ reduction~$[\mathbf{x}]_2$.
\label{def-parity}
\end{definition}

\begin{remark}\label{invariance-of-parity}
We remark that at this level of generality,
the parity depends essentially on the chosen curve $\theta$ and is therefore not invariant under homotopy.
However, if $\mathbf{x}\in\ZZ^{Q(\triang)}$ satisfies the linear congruence $\mathbf{Ax}\equiv 0$ mod~$2$,
then the parity~$\parity_\theta(\mathbf{x})$ only depends on the \nword{\FF_2}{homology} class of $\theta$.
In order to see this, it suffices to consider two curves~$\theta_1$ and $\theta_2$ which are
identical except near an edge~$E$ which they pass on opposite sides, as shown in Figure~\ref{fig:dragging}.
Denote by $\tet_1$ and $\tet_2$ the tetrahedra through which the curves enter and leave the neighbourhood
of $E$. Then
\[
    G_2(\theta_1, \square)-G_2(\theta_2,\square) = G_2(E,\square)+T(\tet_1,\square)+T(\tet_2,\square),
\]
where for a tetrahedron $\tet$ of $\triang$ and a quad $\square\in Q(\triang)$ we set
\[
    T(\tet,\square) =
    \begin{cases}
    1 & \text{if}\ \square\subset\tet,\\
    0 & \text{if}\ \square\not\subset\tet,
    \end{cases}
    \qquad
    T(\tet,\square)\in\FF_2.
\]

Since the `tetrahedral solutions modulo~$2$'
$\bigl(T(\tet,\square)\bigr)_{\square\in Q(\triang)}$
as well as the gluing constraint vector~$\bigl(G_2(E,\square)\bigr)_{\square\in Q(\triang)}$
are elements of $\Imag\mathbf{A}_2^\transp$,
it follows that for every $[\mathbf{x}]_2\in\Ker\mathbf{A}_2$ we have
\[
    \parity_{\theta_1}([\mathbf{x}]_2) - \parity_{\theta_2}
    = \sum_{\square\in Q(\triang)}\bigl(G_2(E,\square)+T(\tet_1,\square)+T(\tet_2,\square)\bigr)\,[\mathbf{x}(\square)]_2
    = 0.
\]
\end{remark}
\begin{figure}[tb]
    \centering
    \parbox[b][][t]{0.27\columnwidth}{
    \begin{tikzpicture}
        \node[anchor=south west,inner sep=0] (image) at (0,0)
            {\includegraphics[width=0.25\columnwidth]{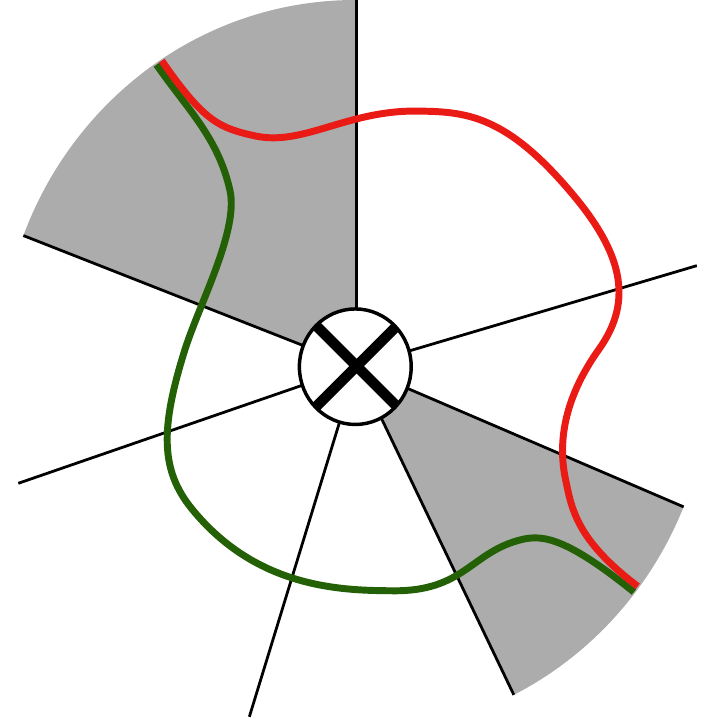}};
        \begin{scope}[x={(image.south east)},y={(image.north west)}]
            \node[anchor=south west] at (0.51, 0.55)  {$E$};
            \node[anchor=south west] at (0.7, 0.8) {$\theta_1$};
            \node[anchor=north east] at (0.3, 0.3) {$\theta_2$};
            \node[anchor=south] at (0.15, 0.65) {$\tet_1$};
            \node[anchor=south] at (0.75, 0.1) {$\tet_2$};
        \end{scope}
    \end{tikzpicture}
    }%
    \parbox[b][][t]{0.7\columnwidth}{
    \caption[Cross-sectional view of an edge neighbourhood with two curves passing on opposite sides]{%
    Cross-sectional view of a neighbourhood of an edge~$E$ of the triangulation~$\triang$ where
    two curves $\theta_1$ and $\theta_2$ deviate from one another.
    The tetrahedra~$\tet_1$ and $\tet_2$ through which the curves enter and leave
    the neighbourhood of $E$ are shaded.
    }\label{fig:dragging}
    }
\end{figure}
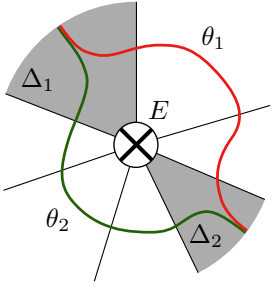

\begin{remark}\label{two-curves}
Whenever $\theta$ and $\psi$ are two closed curves in general position with respect to $\triang$ and
without backtracking, then $\parity_\theta\bigl([l_\psi]_2\bigr) = 0$.
Indeed, we may compute
\begin{equation}\label{eq:LTDs-have-even-parity}
    \parity_\theta\bigl([l_\psi]_2\bigr)
    = \sum_{\square\in Q(\triang)}\mkern-10mu G_2(\theta,\square)[l_\psi]_2(\square)
    \overset{\text{\eqref{Kronecker-pairing-gamma'2}}}{=\joinrel=\joinrel=}
    \gamma'_2\bigl(g_2([\theta])\bigr)\bigl([\psi]\bigr)
    = 0,
\end{equation}
the last equality being a consequence of the exactness of the sequence~\eqref{Neumanns-exact-mod2}.
\end{remark}

The following is one of the main theorems from \cite{neumann1990}. A generalisation of it, proved in the next section, will be a key step in our proof.

\begin{lemma}[Neumann {\cite[Lemma~6.1]{neumann1990}}]\label{neumann-6.1}
Suppose that every component of $\partial M$ is a torus. Then there exists an element
$\eta\in\ZZ^{Q(\triang)}$ satisfying:
\begin{enumerate}[1.]
\item $\displaystyle\eta(\square) + \eta(\square') + \eta(\square'') = 1$ in every tetrahedron,
\item $\displaystyle\mathbf{G}\eta = (2,2,\dotsc,2)^\transp$,
\item $\displaystyle\mathbf{G}_\partial \eta = 0$ in every boundary torus,
\item $\displaystyle\parity_\theta(\eta)=0$ for every closed curve~$\theta\subset M$ in
general position with respect to $\triang$ and without backtracking.
\end{enumerate}
Moreover, any such $\eta$ is unique up to $\mathbf{L^\transp}(\ZZ^{\edges(\triang)})$.
\end{lemma}
Note that while conditions~1--3
simply mean that $\eta$ is a peripherally trivial integer-valued angle structure
in the sense of Definition~\ref{def:Zvals}, condition~4
is somewhat more restrictive and not usually considered in the context of angle structures.
Since we are interested in the spaces of circle-valued angles,
we are now going to introduce analogous structures over $\FF_2$.
\begin{definition}
A \emph{\nword{\FF_2}{valued} angle structure with even parity} on an ideal triangulation~$\triang$
is a vector $\alpha\in\FF_2^{Q(\triang)}$ satisfying the following linear equations over $\FF_2$:
\begin{enumerate}[1.]
\item $\displaystyle \alpha(\square) + \alpha(\square') + \alpha(\square'') = 1$
for all $\square\in Q(\triang)$;
\item $\displaystyle\parity_\theta(\alpha)=0$ for every closed curve~$\theta$
in normal position with respect to $\triang$ and without backtracking.
\end{enumerate}
The set of \nword{\FF_2}{valued} angle structures with even parity on $\triang$
is denoted by $\AFpar(\triang)$.
\end{definition}
Note that condition~2 in the above definition must hold in particular for
null-homologous curves encircling each edge of the triangulation $\triang$
(as in Figure~\ref{fig:coboundary-ltd}, right),
so that $\mathbf{G}_2\alpha = 0$ for every $\alpha\in\AFpar$.
Similarly, condition~2 applied to a peripheral curve~$\theta$ implies
the vanishing of the (\nword{\FF_2}{valued}) peripheral angle-holonomy of $\alpha$.
Hence, for any $\alpha\in\AFpar(\triang)$, we obtain a peripherally trivial
\nword{\{\pm 1\}}{valued} angle structure~$\omega\in\SAS_0(\triang)$ given by
$\omega(\square) = (-1)^{\alpha(\square)}$.

\subsection{Existence theorem for integer-valued pseudo-angle structures}
\label{pseudoangle-existence}
The goal of this section is to establish the following theorem, which
is a generalisation of Neumann's result stated as Lemma~\ref{neumann-6.1} above.

\begin{theorem}
\label{generalization-of-6.1}
Let $\triang$ be an oriented ideal triangulation with $N$ tetrahedra
and $k>0$ ends with toroidal cross-sections~$T_1,\dotsc,T_k$.
In each boundary torus~$T_j$, choose a pair of oriented peripheral curves
$(\mu_j,\lambda_j)$ satisfying $\intersection(\mu_j,\lambda_j)=1$.
Then for every even vector~$d\in2\ZZ^{2k}$ and for any even vector $c\in 2\ZZ^{\edges(\triang)}$
satisfying $p(c-\mathbf{2})=0$, there exists an element $\eta\in\ZZ^{Q(\triang)}$
for which all of the following conditions hold:
\begin{enumerate}[1.]
\item $\displaystyle\eta(\square) + \eta(\square') + \eta(\square'') = 1$ in every tetrahedron,
\item $\displaystyle\mathbf{G}\eta = c$,
\item $\displaystyle\mathbf{G}_\partial \eta = d$,
\item $\displaystyle\parity_\theta(\eta)=0$ for every closed curve~$\theta$ in general position with
respect to $\triang$ and without backtracking.
\end{enumerate}
Moreover, any two such vectors~$\eta$ differ by an element
of $\mathbf{L^\transp}(\ZZ^{\edges(\triang)})$.
\end{theorem}
Observe that Lemma~\ref{neumann-6.1} covers the special case of
$c=\mathbf{2}=(2,\dotsc,2)^\transp$ and $d=0$, which corresponds to the geometric
component~$\SASdistinguished(\triang)$ under the exponential map.
In order to prove the theorem for general $c$ and $d$, we first establish a simpler
version of the statement.

\begin{lemma}\label{homogenous-generalization-of-neumanns6.1}
For every $c\in 2\ZZ^{\edges(\triang)}\cap \Ker p$
and every $d\in2\ZZ^{2k}$,
there exists an $\eta\in\ZZ^{Q(\triang)}$ such that:
\begin{enumerate}[1.]
\item $\displaystyle\eta(\square) + \eta(\square') + \eta(\square'') = 0$ in every tetrahedron,
\item $\displaystyle\mathbf{G}\eta = c$,
\item $\displaystyle\mathbf{G}_\partial \eta = d$,
\item $\displaystyle\parity_\theta(\eta)=0$ for every closed curve~$\theta$ in general position with
respect to $\triang$ and without backtracking.
\end{enumerate}
Moreover, any two such elements~$\eta$ differ by an element
of $\mathbf{L^\transp}(\ZZ^{\edges(\triang)})$.
\end{lemma}

\begin{proof}[Proof of Lemma~\ref{homogenous-generalization-of-neumanns6.1}]
Recall from Proposition~\ref{computing-homology-of-hat}
that $\Ker p / \Imag_\ZZ\mathbf{L}\cong H_1(\hat{\triang};\FF_2)$.
Whenever $c\in\Ker p$ is an even vector, $c$ represents the trivial
class in $H_1(\hat{\triang};\FF_2)$, so $c\in\Imag_\ZZ\mathbf{L}$.
Let $\mathbf{x}\in\ZZ^{Q(\triang)}$ be such that $\mathbf{Lx}=c$ and write
$\mathbf{x}=(x,x',x'')^\transp$ using the ordering~\eqref{quad-ordering}.
Then the vector $r:=(x'-x'', x''-x, x-x')^\transp$ satisfies
$r(\square)+r(\square')+r(\square'')=0$ in every tetrahedron and moreover
\[
    \mathbf{G}r = [G\;|\;G'\;|\;G'']
        \left(\begin{smallmatrix}
            x'\phantom{'}-x''\\
            x''-x\phantom{''}\\
            x\phantom{''}-x'\phantom{'}
        \end{smallmatrix}\right)
    = (G''-G')x + (G-G'')x' + (G'-G)x'' = \mathbf{Lx} = c.
\]
In other words, $r$ satisfies conditions 1 and 2.

Suppose now that $\theta$ is a curve in normal position with respect
to $\triang$ and without backtracking. As $c$ is an even vector,
conditions 1 and 2 imply that $r$ satisfies
$\mathbf{A}r\equiv 0$ mod~$2$, so Remark~\ref{invariance-of-parity} guarantees
that the parity of $r$ along $\theta$ only depends on the \nword{\FF_2}{homology}
class~$[\theta]$ of the curve~$\theta$.
Hence, the parity of $r$ can be viewed as a cohomology class,
\begin{align*}
    &\parity_{(-)}(r) \colon [\theta]\mapsto \parity_{\theta}(r),\\
    &\parity_{(-)}(r) \in \Hom\bigl(H_1(M;\FF_2),\; \FF_2\bigr) = H^1(M;\FF_2).
\end{align*}
By Proposition~\ref{neumanns-exact-sequences},
$\parity_{(-)}(r)=\gamma'_2\iota(\mathbf{w})$ for some representing
element~$\mathbf{w}\in \Ker_\ZZ\mathbf{L}$.
Writing $\mathbf{w}=(w,w',w'')$, using \eqref{Kronecker-pairing-gamma'2} and
then re-indexing the sum, we get
\begin{align}
    \parity_\theta(r)=
    \bigl({\gamma'_2}(\mathbf{w})\bigr)([\theta]) &=
    \sum_{\square\in Q(\triang)} \notag
        \left[\mathbf{w}(\square)\right]_2 \bigl(G_2(\theta,\square'') - G_2(\theta,\square')\bigr)\\
    &= \label{relating-gamma'_2-to-parity}
    \sum_{\square\in Q(\triang)}
        G_2(\theta,\square) \bigl[\mathbf{w}(\square')-\mathbf{w}(\square'')\bigr]_2\\
    &= \parity_\theta\bigl(w'-w'', w''-w, w-w'\bigr).\notag
\end{align}
Since $\theta$ is arbitrary, the vector~$(w'-w'', w''-w, w-w')^\transp\in\Ker\mathbf{A}$ has the
same parity as $r$.
Therefore, $s := r - (w'-w'', w''-w, w-w')^\transp$ has even parity along all
normal curves without backtracking.
In other words, $s$ satisfies conditions 1, 2 and 4.

We shall now deal with condition 3 which concerns peripheral angle-holonomies.
We consider first the case of a single cusp ($k=1$)
with oriented curves $\mu$, $\lambda$ satisfying
$\intersection(\mu,\lambda)=1$ and with $d=(d_\mu, d_\lambda)^\transp$ for
arbitrarily fixed even integers~$d_\mu,d_\lambda$.
Since the parity condition holds for the curves~$\mu$ and $\lambda$, we have
$\mathbf{G_\partial}s = (a_\mu, a_\lambda)^\transp$
for some even integers $a_\mu, a_\lambda$.
Using the leading-trailing deformations $l_\mu$ and $l_\lambda$ and
Lemma~4.4 of Futer--Gu\'eritaud~\cite{futer-gueritaud},
we see
that the integer vector
\begin{equation}\label{eq:eta-that-works-for-the-lemma}
    \eta := s + \frac{a_\mu-d_\mu}{2} l_\lambda - \frac{a_\lambda-d_\lambda}{2} l_\mu
\end{equation}
satisfies $\mathbf{G_\partial}\eta = (d_\mu, d_\lambda)^\transp$
and thus all of the conditions 1--3.
In the general case, when $k>1$, such adjustments can be performed independently
for each toroidal end of $\triang$, since curves lying in different components
have zero intersection.

We must now prove that the integer vector~$\eta$
defined in \eqref{eq:eta-that-works-for-the-lemma}
still has even parity along all curves in general position
with respect to $\triang$ and without backtracking.
To this end, note that $\eta$ differs from $s$ by an integer linear combination
of vectors of the form $l_\psi$ for
$\psi\in\{\mu_j,\lambda_j\}_{j=1}^k$.
Hence, Remark~\ref{two-curves} implies that $\eta$ has even parity
and thus satisfies all conditions~1--4.

Lastly, any two vectors $\eta$ satisfying conditions 1--4 with the same $c$ and $d$
must differ by an element satisfying 1--4 with $c=0$ and $d=0$; that the set of such elements
is exactly $\Imag_\ZZ\mathbf{L}^\transp$ follows from Lemma~\ref{neumann-6.1}.
\end{proof}

\begin{proof}[Proof of Theorem~\ref{generalization-of-6.1}]
Let $c$ and $d$ be even vectors satisfying the conditions of the theorem.
If $\eta$ is given by Lemma~\ref{neumann-6.1} and $\tilde{\eta}$ is
obtained by an application of Lemma~\ref{homogenous-generalization-of-neumanns6.1} with
$\tilde{c} = c-\mathbf{2}\in\Ker p$ and $\tilde{d}=d$, then
$\eta+\tilde{\eta}$ satisfies all conditions 1--4 of Theorem~\ref{generalization-of-6.1}.
\end{proof}

\subsection{Truncated triangulations and rectangular cocycles}
\label{truncated-triangulations}
\hyphenation{ma-ni-fold}
We are now going to recall two well-known cell decompositions of $M$ 
 induced by
the triangulation $\triang$.
The first of these decomposition is the (polyhedral) cell complex~$\triang_0$
whose \nword{3}{cells} are \emph{truncated tetrahedra} which result
from slicing off the ideal vertices of the tetrahedra of $\triang$.
In particular, the boundary 
is triangulated into the cutoff
triangles of the truncated tetrahedra; we denote
this boundary triangulation $\partial\triang_0$ and treat it as a subcomplex of $\triang_0$.

We shall also need the cell complex~$\triang_{00}$ consisting of
\emph{doubly truncated tetrahedra} together
with solid cylinders around the edges of $\triang$.
The doubly truncated tetrahedra of~$\triang_{00}$ result from the removal of
neighbourhoods of vertices and edges of the tetrahedra of $\triang$.
The cell complex~$\triang_{00}$ has three types of \nword{1}{cells}, which
are shown in Figure~\ref{fig:doubly-truncated}.
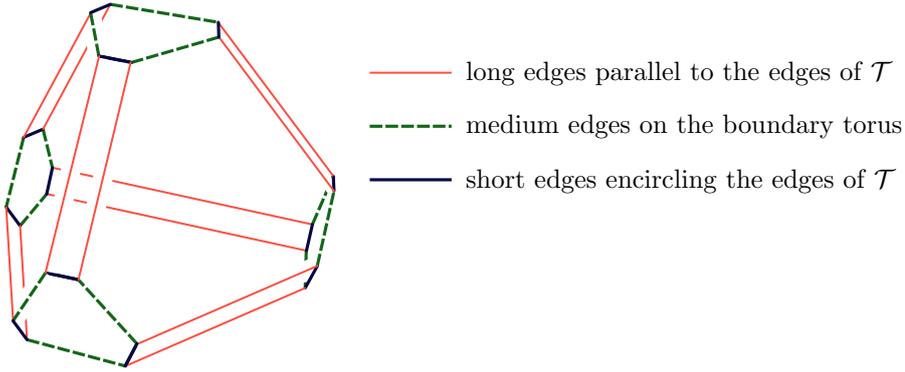
\begin{figure}[tb]
    \centering
    \begin{tikzpicture}
        \node[anchor=south west,inner sep=0] (image) at (0,0)
            {\includegraphics[width=0.4\columnwidth]{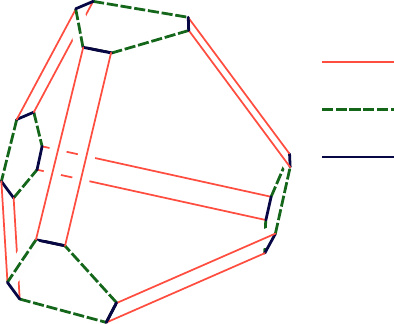}};
        \begin{scope}[x={(image.south east)},y={(image.north west)}]
            \node[anchor=west] at (1.01, 0.805)  {long edges parallel to the edges of $\triang$};
            \node[anchor=west] at (1.01, 0.66)  {medium edges on the boundary torus};
            \node[anchor=west] at (1.01, 0.51)  {short edges encircling the edges of $\triang$};
        \end{scope}
    \end{tikzpicture}
    \caption{%
    A doubly truncated tetrahedron in $\triang_{00}$ and the three types of its edges.
    }\label{fig:doubly-truncated}
\end{figure}
There are also four types of faces (\nword{2}{cells} in the cell complex~$\triang_{00}$):
\emph{large hexagonal faces} contained in the faces of the original triangulation $\triang$,
\emph{boundary hexagons} at the vertex truncation locus, \emph{rectangles} along the edges of the
tetrahedra of $\triang$, and \emph{edge discs} capping off the cylinders about the edges of $\triang$.
The boundary complex~$\partial\triang_{00}$ contains the edge discs and boundary hexagons
as its \nword{2}{cells}.

\begin{definition}[Rectangle map]
Let $\alpha\in\Ker\mathbf{G}_2\subset\FF_2^{Q(\triang)}$.
We define $\rect(\alpha)\in C^2(\triang_{00}, \partial\triang_{00}; \FF_2)$
to be the \nword{2}{cocycle} satisfying
$\bigl(\rect(\alpha)\bigr)(r) = \alpha(\square)$ whenever $r$ is a rectangle
parallel to the normal quadrilateral~$\square$, and $\bigl(\rect(\alpha)\bigr)(f)=0$
for any \nword{2}{cell}~$f$ which is not a rectangle.
\end{definition}
Observe that $\rect(\alpha)$ takes equal values on opposite rectangles, so
$\delta^2\bigl(\rect(\alpha)\bigr)$ contributes $0$ to the interiors of the
doubly truncated tetrahedra.
Moreover, the assumption that $\alpha\in\Ker\mathbf{G}_2$ implies that the coboundary
of $\rect(\alpha)$ also vanishes on the solid cylinders about the edges of $\triang$.
Therefore, $\rect(\alpha)$ is a relative \nword{2}{cocycle}
and we have a \nword{\FF_2}{linear} map
\begin{equation}\label{def-rect}
    \rect: \Ker\mathbf{G}_2 \to Z^2(\triang_{00}, \partial\triang_{00}; \FF_2),
    \qquad
    \alpha \mapsto \rect(\alpha).
\end{equation}

\subsubsection{The fanning construction and its inverse}
We are now going to establish a few elementary properties
of cocycles in the image of the map~$\rect$ of \eqref{def-rect}.
In Theorem~\ref{homogenous-preimage} below, we show that
cocycles of this form suffice to represent all cohomology
classes in $H^2(M,\partial M;\FF_2)$.

\begin{definition}[The fanning construction]
For a large hexagonal face~$F$ of $\triang_{00}$, denote by
$b_F\in C^1(\triang_{00}, \partial\triang_{00}; \FF_2)$ the cochain
whose value on a long edge is the (modulo 2) number of times
this long edge is incident to $F$ and which vanishes everywhere else.
For an arbitrary relative
\nword{2}{cochain}~$\sigma\in C^2(\triang_{00}, \partial\triang_{00}; \FF_2)$,
we set
\begin{equation}\label{def-fanning}
    Y(\sigma) = \sigma + \sum_{F\;:\;\sigma(F)=1} \delta^1(b_F),
\end{equation}
where the sum is over all large hexagonal faces in
the support of $\sigma$. See Figure~\ref{fig:fanning} for an illustration.
\end{definition}

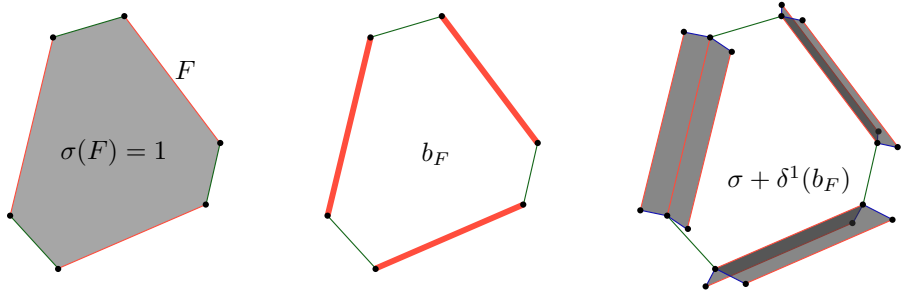
\begin{figure}[tb]
    \centering
    \begin{tikzpicture}
        \node[anchor=south west,inner sep=0] (image) at (0,0)
            {\includegraphics[width=0.8\columnwidth]{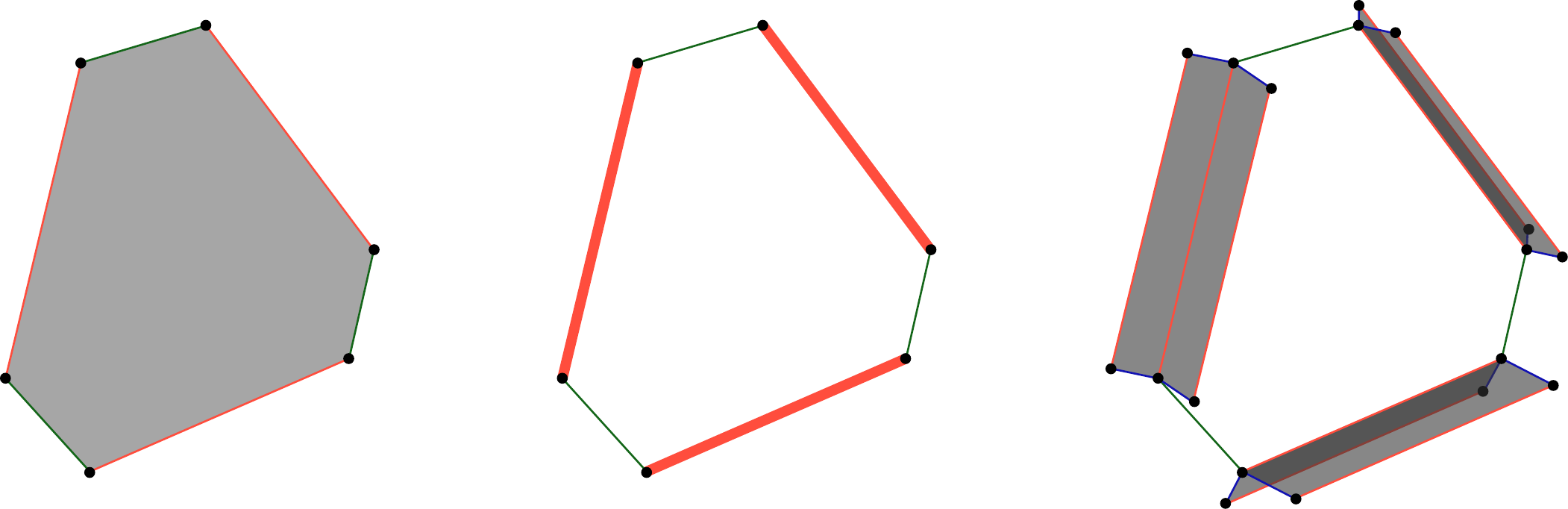}};
        \begin{scope}[x={(image.south east)},y={(image.north west)}]
            \node[anchor=south] at (0.2, 0.7) {$F$};
            \node[anchor=south] at (0.12, 0.4) {$\sigma(F)=1$};
            \node[anchor=south] at (0.48, 0.4) {$b_F$};
            \node[anchor=south] at (0.875, 0.3) {$\sigma + \delta^1(b_F)$};
        \end{scope}
    \end{tikzpicture}
    \caption{
    \textsc{Left:} A large hexagonal face~$F$ with $\sigma(F)=1$.
    \textsc{Centre:} The \nword{1}{cochain} $b_F$ is supported on the three long edges of the face $F$
        (shown in bold).
    \textsc{Right:} $Y(\sigma) = \sigma + \delta^1(b_F)$ vanishes on $F$ but may have a
        non-zero value on rectangular \nword{2}{cells} of $\triang_{00}$.
    }\label{fig:fanning}
\end{figure}

When $\sigma$ is a cocycle, equation~\eqref{def-fanning} shows that
$Y(\sigma)$ is cohomologous to $\sigma$.
It follows that all cohomology classes in $H^2(M,\partial M;\FF_2)$ are representable
with cocycles in $Z^2(\triang_{00}, \partial\triang_{00}; \FF_2)$ supported on the rectangles
only.

\begin{remark}
Suppose that $\sigma\in C^2(\triang_{00}, \partial\triang_{00}; \FF_2)$ is supported on large
hexagonal faces only. Let $F_1$ and $F_2$ be two large hexagonal faces incident to the same
doubly truncated tetrahedron and let $r$ be the rectangle between $F_1$ and $F_2$.
Then $\bigl(Y(\sigma)\bigr)(r) = \sigma(F_1)+\sigma(F_2)$.
\label{rem:facial-difference}
\end{remark}

\begin{remark}\label{linearity-of-Y}
If $\sigma_1, \sigma_2\in C^2(\triang_{00}, \partial\triang_{00}; \FF_2)$ are cochains
supported on large hexagonal faces only, then $Y(\sigma_1+\sigma_2) = Y(\sigma_1)+Y(\sigma_2)$.
\end{remark}

We now describe a partial inverse to the fanning construction.
Let $\mathcal{S}$ be the vector subspace of $\FF_2^{Q(\triang)}$ cut out by the equations
\begin{equation}\label{def-of-subspace-S}
\left\{
    \begin{aligned}
        s(\square) + s(\square') + s(\square'') &= 0\ \text{for all}\ \square,\\
        \parity_\theta(s) &= 0\ \text{for every curve~$\theta$ in normal position.}
    \end{aligned}\right.
\end{equation}
Observe that every $s\in\mathcal{S}$ satisfies $\mathbf{G}_2s=0$, which
is seen by considering small loops encircling each of the edges of $\triang$.
In particular, the map~$\rect$ of \eqref{def-rect} is defined on $\mathcal{S}$.

\begin{lemma}
For any $s\in\mathcal{S}$, there exists a cocycle $\sigma\in Z^2(\triang_{00}, \partial\triang_{00}; \FF_2)$
supported on the large hexagonal faces only and satisfying $Y(\sigma)=\rect(s)$.
\label{defanning}
\end{lemma}

\begin{proof}
Consider a fixed element $s\in\mathcal{S}$.
Choose any large hexagonal face $F_0$ in $\triang_{00}$ and set $\sigma(F_0)\in\FF_2$ arbitrarily.
We are going to continue $\sigma$ to other large hexagonal faces using the following rule.
When $F_1$ and $F_2$ are two large hexagonal faces incident to the same doubly truncated tetrahedron
and $r$ is the rectangle between $F_1$ and $F_2$, then $\sigma(F_1)+\sigma(F_2)=\bigl(\rect(s)\bigr)(r)$.

Observe that the equation $s(\square)+s(\square')+s(\square'')=0$ implies
that the above rule determines a consistent labelling of the large hexagonal faces of each doubly
truncated tetrahedron in $\triang_{00}$ with elements of $\FF_2$.
In order to see that this labelling is globally consistent, consider
an unoriented closed curve~$\theta$ disjoint from the closures of the cylinders
about the edges of $\triang$.
We may assume that $\theta$ is based at the initial face $F_0$ and crosses
large hexagonal faces transversely.
Since $\parity_\theta(s)=0$, the face labelling
continued along $\theta$ agrees with the initial labelling of $F_0$.
In this way, we obtain a well-defined
cochain $\sigma\in C^2(\triang_{00}, \partial\triang_{00}; \FF_2)$
supported on large hexagonal faces only.
What is more, Remark~\ref{rem:facial-difference} implies that
whenever $r$ is a rectangle between two large hexagonal faces $F_1$ and $F_2$,
we must have $Y(\sigma)(r) = \sigma(F_1)+\sigma(F_2)=\bigl(\rect(s)\bigr)(r)$.
Therefore, $Y(\sigma)=\rect(s)$.
Since $\sigma$ differs from $Y(\sigma)$ by a coboundary and $\rect(s)$ is a cocycle,
$\sigma$ is also a cocycle, as required.
\end{proof}

\begin{remark}\label{ltds-mod-2-contained-in-S}
For any closed curve~$\psi$ in general position with respect to $\triang$ and without backtracking,
the modulo~$2$ leading-trailing deformation~$[l_\psi]_2$
is an element of the space $\mathcal{S}$ defined by~\eqref{def-of-subspace-S}.
Indeed, the equations~$[l_\psi]_2(\square)+[l_\psi]_2(\square')+[l_\psi]_2(\square'')=0$
arise as a direct consequence of~\eqref{curve-ltd-mod-2}.
Moreover, Remark~\ref{two-curves} states that
$\parity_\theta\bigl([l_\psi]_2\bigr)=0$ for any closed curve~$\theta$
in general position with respect to $\triang$ and without backtracking.
\end{remark}

The following lemma provides a cohomological interpretation of
the cocycle~$\rect([l_\theta]_2)$.
\begin{lemma}
Let $\theta$ be a closed curve in general position with respect to $\triang$ and without
backtracking.
Then the cohomology class represented by $\rect([l_\theta]_2)$
in $H^2(M,\partial M;\FF_2)$ coincides with the
Poincar\'e dual of $[\theta]\in H_1(M;\FF_2)$.
\label{rectal-lemma}
\end{lemma}

\begin{proof}
We can assume that the curve $\theta$ misses the solid cylinders in $\triang_{00}$
and only travels through the interiors of the doubly truncated tetrahedra,
cutting transversely across their large hexagonal faces.
Then the Poincar\'e dual of $[\theta]$ is represented by the cellular \nword{2}{cocycle}
$C_\theta\in Z^2(\triang_{00},\partial\triang_{00};\FF_2)$
whose value on a large hexagonal face is the number of times $\theta$ crosses that face,
modulo~$2$.
Applying the fanning construction, we obtain a cocycle~$Y(C_\theta)$
representing the Poincar\'e dual of $[\theta]$ and supported
only on the rectangular \nword{2}{cells} of $\triang_{00}$.
Since we disallow backtracking, upon every visit to a tetrahedron~$\tet$,
the curve~$\theta$ will cross exactly two faces of $\tet$.
Considering one such visit at a time, we denote by $e$ the tetrahedral edge of $\tet$
common to these two faces and let $\square$ be the normal quadrilateral faced by $e$.
By Remark~\ref{rem:facial-difference}, the visit contributes to $Y(C_\theta)$
the value of zero on the two opposing rectangles parallel to $\square$ and the value
$1\in\FF_2$ on the remaining four rectangles in $\tet$.
This agrees with the contribution of the visit to $\rect([l_\theta]_2)$.
Since both cocycles are sums of such contributions over all passages of $\theta$
through the tetrahedra of $\triang_{00}$,
$Y(C_\theta)$ coincides with $\rect([l_\theta]_2)$.
\end{proof}

\begin{corollary}\label{rect-edge-ltds-are-coboundaries}
For any edge~$E$ of the triangulation~$\triang$,
the cocycle~$\rect([l_E]_2)$ represents the trivial class in $H^2(M,\partial M;\FF_2)$.
\end{corollary}
\begin{proof}
We may consider a closed curve~$\varepsilon$ contained in a small neighbourhood
of $E$ and encircling $E$ once, as shown in the right panel of Figure~\ref{fig:coboundary-ltd}.
In this way, $[l_E]_2 = [l_\varepsilon]_2$.
However, the curve~$\varepsilon$ is null-homologous,
so $\rect([l_E]_2)=\rect([l_\varepsilon]_2)$
represents the trivial class in $H^2(M,\partial M;\FF_2)$.
\end{proof}

\begin{theorem}\label{homogenous-preimage}
Let $\mathcal{S}\subset\FF_2^{Q(\triang)}$ be the set of solutions to the
linear equations~\eqref{def-of-subspace-S}.
\begin{enumerate}[(i)]
\item\label{item:rectonto}
Every cohomology class in $H^2(M,\partial M;\FF_2)$ is represented
by a cocycle of the form $\rect(s)$ for some $s\in\mathcal{S}$.
\item\label{item:KerLTD}
For $s\in\mathcal{S}$, the cocycle $\rect(s)$ is zero
in $H^2(M,\partial M;\FF_2)$ if and only if $s\in\Imag\mathbf{L}_2^\transp$.
\end{enumerate}
\end{theorem}

\begin{proof}
In order to prove the first part, choose an arbitrary cohomology
class~$[\sigma]\in H^2(M,\partial M;\FF_2)$ represented by
a cocycle $\sigma\in Z^2(\triang_0, \partial\triang_0; \FF_2)$.
Note that there is a cellular homotopy equivalence of CW-pairs,
\begin{equation}\label{cylinder-collapse}
    h_0 : (\triang_{00},\partial\triang_{00})\xrightarrow{\simeq}(\triang_{0},\partial\triang_{0}),
\end{equation}
which collapses the closures of the cylinders in $\triang_{00}$ into the
corresponding non-boundary edges of $\triang_0$.
Let $\sigma^*=h_0^*(\sigma)\in Z^2(\triang_{00}, \partial\triang_{00}; \FF_2)$
be the pullback of $\sigma$, so that $\sigma^*$ is a \nword{2}{cocycle} supported
on the large hexagonal faces of the doubly truncated tetrahedra of $\triang_{00}$.
Applying the fanning construction, we obtain a
\nword{2}{cocycle}~$Y(\sigma^*)$ supported on rectangles only.
The cocycle condition for $\sigma$ implies that every doubly truncated tetrahedron
will see a non-zero value of $\sigma^*$ on an even number of its large hexagonal faces.
If this number is $0$ or $4$, then Remark~\ref{rem:facial-difference} shows
that $Y(\sigma^*)$ vanishes on all rectangles in the tetrahedron.
If however $\sigma$ takes the value $1\in\FF_2$ on exactly two faces of a
tetrahedron~$\tet$, then $Y(\sigma^*)$ vanishes
on exactly one pair of opposite rectangles in $\tet$
and takes the value $1$ on the four remaining rectangles.
In particular, $Y(\sigma^*)$ takes identical values on opposite rectangles in all cases.
Let $s: Q(\triang) \to \FF_2$ be the vector assigning to every quad~$\square$
the value of $Y(\sigma^*)$ on the rectangles parallel to $\square$,
so that $Y(\sigma^*)=\rect(s)$.
We are now going to show that $s\in\mathcal{S}$.

As observed above, $Y(\sigma^*)$ is either zero on all six rectangles in a
tetrahedron, or it takes the value $1$ on exactly two pairs of opposite rectangles.
This shows that $s(\square)+s(\square')+s(\square'')=0$.
It remains to consider the parity conditions.
Let $\theta$ be a curve in normal position with respect to $\triang$ and without backtracking.
Consider first a single crossing between $\theta$ and a large hexagonal face~$F$ of $\triang_{00}$
satisfying $\sigma^*(F)=1$.
As shown on the right panel of Figure~\ref{fig:fanning}, the contribution of this crossing
to the parity of $s$ along $\theta$ will come from a pair of
rectangles adjacent to $F$, one in each of the tetrahedra separated by $F$.
Since $F$ contributes $1$ to the value of $Y(\sigma^*)$ on both
rectangles, the total contribution of the two rectangles to the parity along $\theta$
is always $0\in\FF_2$.
In general, the parity of $s$ along $\theta$ is a sum of such contributions, so it always vanishes,
and part~(i) is established.

In order to prove part~(ii), suppose first that $s\in\mathcal{S}$ is such that $\rect(s)$
represents the trivial class in $H^2(M,\partial M;\FF_2)$.
By Lemma~\ref{defanning}, there exists a relative cocycle
$\sigma\in Z^2(\triang_{00}, \partial\triang_{00}; \FF_2)$
supported on the large hexagonal faces only and such that $Y(\sigma)=\rect(s)$.
Hence, $\sigma = h_0^*(\Sigma)$, where the value of the relative
cocycle~$\Sigma\in Z^2(\triang_{0}, \partial\triang_{0}; \FF_2)$
on each interior face in $\triang_0$ coincides with
the value of $\sigma$ on the corresponding large hexagonal face in $\triang_{00}$.

For an edge~$E$ of $\triang$, let $E_0$ be the
corresponding \nword{1}{cell} in $\triang_0$.
Denote by $\chi_E\in C^1(\triang_{0}, \partial\triang_{0}; \FF_2)$ the \nword{1}{cochain}
taking the value $1\in\FF_2$ on $E_0$ and vanishing everywhere else.
Since the cochains $\left\{\chi_E\right\}_{E\in\edges(\triang)}$
span the space $C^1(\triang_{0}, \partial\triang_{0}; \FF_2)$
and since $\Sigma$ is a coboundary,
it follows that there exists a subset $\mathcal{U}\subset\edges(\triang)$
such that
\(\displaystyle
    \Sigma = \sum_{E\in\mathcal{U}} \delta^1(\chi_E)
\),
which implies that
\begin{equation}\label{edge-decomp-of-coboundary}
    \rect(s) = Y\bigl(h_0^*(\Sigma)\bigr)
    = \sum_{E\in\mathcal{U}} Y\bigl(h_0^*\bigl(\delta^1(\chi_E)\bigr)\bigr),
\end{equation}
where the second equality uses Remark~\ref{linearity-of-Y}.
\begin{figure}
    \centering
    \begin{tikzpicture}
        \node[anchor=south west,inner sep=0] (image) at (0,0)
            {\includegraphics[width=0.95\columnwidth]{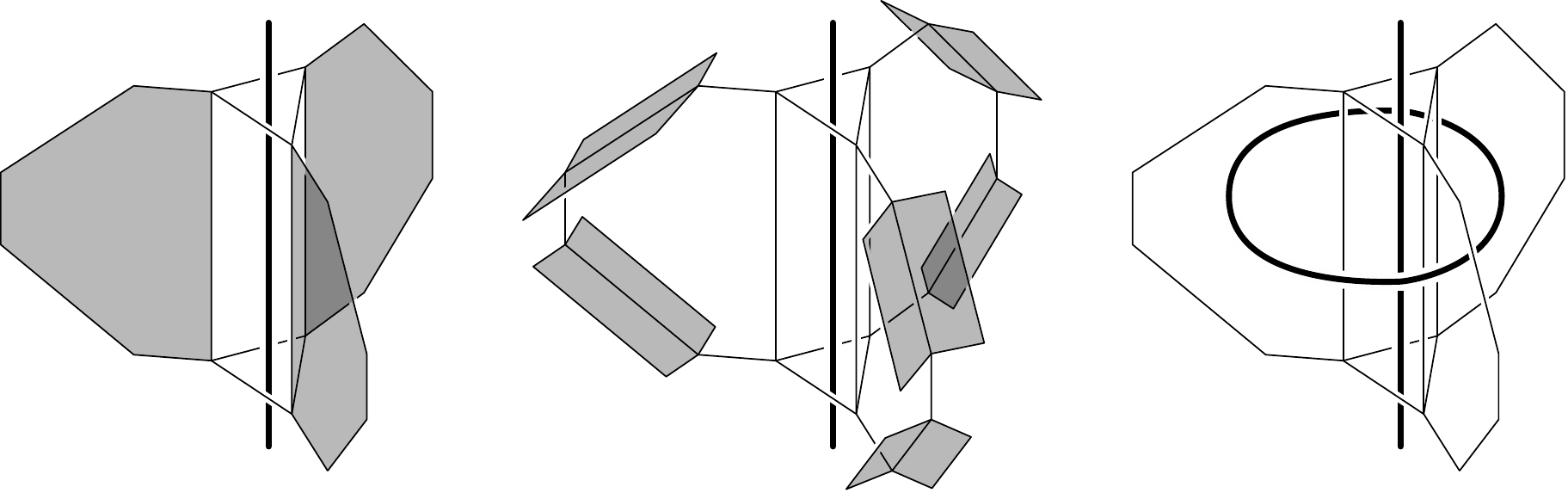}};
        \begin{scope}[x={(image.south east)},y={(image.north west)}]
            \node[anchor=east] at (0.17, 0.95) {$E$};
            \node[anchor=east] at (0.53, 0.95) {$E$};
            \node[anchor=east] at (0.891, 0.95) {$E$};
            \node[anchor=north east] at (0.8, 0.5) {$\varepsilon$};
        \end{scope}
    \end{tikzpicture}
    \caption[A figure with 3 panels.]{\textsc{Left:}
    The cochain $h_0^*\bigl(\delta^1(\chi_E)\bigr)$ receives a contribution of $1\in\FF_2$ to each of
    the large hexagonal faces of $\triang_{00}$ contained in the faces of $\triang$
    incident to the edge~$E$.
    \textsc{Centre:}
    Applying the fanning construction, we see that $Y\bigl(h_0^*\bigl(\delta^1(\chi_E)\bigr)\bigr)$
    coincides with $\rect([l_E]_2)$, where $l_E$ is the leading-trailing deformation about the edge~$E$.
    \textsc{Right:}
    A closed curve~$\varepsilon$ fully contained in the tetrahedra incident to the the edge~$E$
    and encircling $E$ once.
    }\label{fig:coboundary-ltd}
\end{figure}
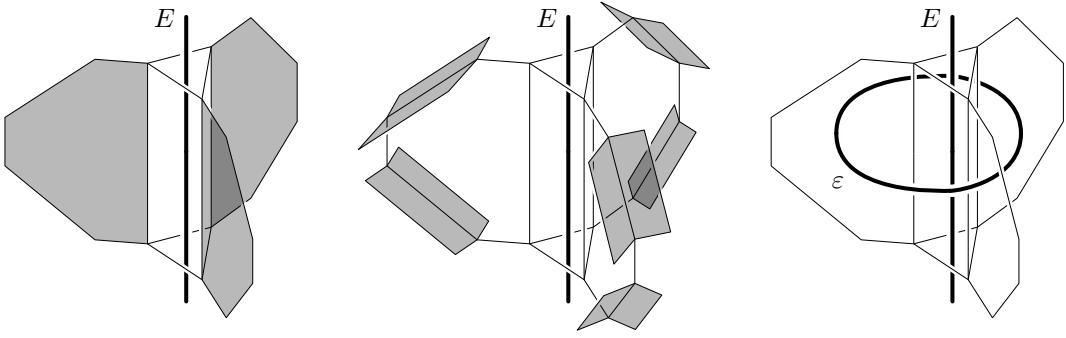%
We shall now consider a single edge $E\in\mathcal{U}$.
For each face~$F$ of $\triang$,
the value of the cochain $h_0^*\bigl(\delta^1(\chi_E)\bigr)$
on the large hexagonal cell inside $F$ is equal to the (mod~$2$) number of edges of $F$
in the edge class of $E$.
This is illustrated in Figure~\ref{fig:coboundary-ltd}.
Comparing the left and middle panels of the figure makes it
clear that $Y\bigl(h_0^*\bigl(\delta^1(\chi_E)\bigr)\bigr) = \rect([l_E]_2)$,
where $[l_E]_2$ is the mod~$2$ reduction of the leading-trailing deformation
about the edge~$E$.
In other words, \eqref{edge-decomp-of-coboundary} implies that
$s\in\linspan_{\FF_2}\{[l_E]_2: E\in\edges(\triang)\}=\Imag\mathbf{L}_2^\transp$.

In order to finish the proof of part~(ii), suppose that $s\in\Imag\mathbf{L}_2^\transp$.
Hence, $s$ is a linear combination of vectors of the form~$[l_E]_2$,
where $E$ is an edge of $\triang$.
Corollary~\ref{rect-edge-ltds-are-coboundaries} now implies
that $\rect(s)$ is a coboundary.
\end{proof}

\subsection{Definition of the obstruction maps}
\label{obstruction-map-defn}
In this section, we are going to define the maps $\Phi$ and $\Phi_0$ occurring in the statement
of Theorem~\ref{obstruction-theorem}.
As in the preceding section, we are representing the pair~$(M, \partial M)$ by the CW-pair
$(\triang_{00}, \partial\triang_{00})$.

\begin{definition}\label{def-Phi}
We define
\begin{equation*}
    \mathcal{R}_0: \AFpar(\triang) \to H^2(M, \partial M;\FF_2),\qquad
    \mathcal{R}_0(\alpha) = [\rect(\alpha)].
\end{equation*}
Likewise, using the inclusion
$Z^2(\triang_{00}, \partial\triang_{00}; \FF_2)\subset Z^2(\triang_{00}; \FF_2)$,
we set
\begin{equation*}
    \mathcal{R}:\AFpar(\triang) \to H^2(M ;\FF_2),\qquad
    \mathcal{R}(\alpha) = [\rect(\alpha)].
\end{equation*}
Moreover, we define
$\Phi:\SAS(\triang)\to H^2(M;\FF_2)$ and
$\Phi_0:\SAS_0(\triang)\to H^2(M,\partial M;\FF_2)$
as the unique continuous extensions of the maps defined by
$\Phi\bigl((-1)^\alpha\bigr)=\mathcal{R}(\alpha)$
and $\Phi_0\bigl((-1)^\alpha\bigr)=\mathcal{R}_0(\alpha)$
for $\alpha$ in $\AFpar(\triang)$.
\end{definition}

Note that the above definition implies that $\mathcal{R}=\Xi\circ\mathcal{R}_0$, where
\[
    \Xi: H^2(M,\partial M; \FF_2) \to H^2(M;\FF_2)
\]
is the natural map occurring in the long exact sequence of the pair~$(M,\partial M)$
in \nword{\FF_2}{cohomology}.
Similarly, the equality $\Phi = \Xi\circ\Phi_0$ holds
on the set of \Sval\ angle structures of the form~$(-1)^\alpha$
for $\alpha\in\AFpar(\triang)$.

\begin{proposition}\label{Phis-welldef}
The maps $\Phi$ and $\Phi_0$ are well-defined. In other words:
\begin{enumerate}[(i)]
\item\label{item:special-points-are-abundant}
    Every connected component of $\SAS_0(\triang)$, and thus
    every connected component of $\SAS(\triang)$, contains a $\{\pm 1\}$-valued angle structure
    of the form $\omega=(-1)^\alpha$ for some $\alpha\in\AFpar(\triang)$.
\item\label{item:Phis-welldef-on-components}
    If $\alpha_1,\alpha_2\in\AFpar(\triang)$ are such that
    $\omega_1 = (-1)^{\alpha_1}$ and $\omega_2 = (-1)^{\alpha_2}$
    lie in the same connected component of
    $\SAS_0(\triang)$, then $\mathcal{R}_0(\alpha_1)=\mathcal{R}_0(\alpha_2)$.
    Likewise, if $\omega_1$ and $\omega_2$ lie in the same connected component
    of $\SAS(\triang)$, then $\mathcal{R}(\alpha_1)=\mathcal{R}(\alpha_2)$.
\end{enumerate}
\end{proposition}
\begin{proof}
In order to prove part~\eqref{item:special-points-are-abundant},
it will be helpful to consider the defining equations for the circle-valued
angle space~$\SAS_0(\triang)$ as equations in real variables modulo $2\pi$,
via the identification $\RR/2\pi\ZZ \cong S^1$ established by the exponential
map~$\alpha\mapsto\exp(i\alpha)$.
Let $C$ be a connected component of $\SAS_0(\triang)$.
Lifting the defining equations of $C$, we see that
there exists an even vector $c\in2\ZZ^{\edges(\triang)}$ and
an even vector $d\in2\ZZ^{2k}$ such that $C$ is the image under the
exponential map of the solution set in $\RR^{Q(\triang)}$ of the equations
\begin{equation}\label{generalized-angle-equations}
\left\{
    \begin{aligned}
        \alpha(\square) + \alpha(\square') + \alpha(\square'') &= \pi\ \text{for all}\ \square,\\
        \mathbf{G}\alpha &= \pi c,\\
        \mathbf{G}_\partial \alpha &= \pi d.
    \end{aligned}\right.
\end{equation}
Moreover, by \cite[Theorem~4.1]{neumann1990}, the vector $c$ satisfies $p(c-\mathbf{2})=0$.
By Theorem~\ref{generalization-of-6.1}, there exists an integer vector~$\eta\in\ZZ^{Q(\triang)}$
such that $\pi\eta$ satisfies equations~\eqref{generalized-angle-equations}
and additionally $\eta$ has even parity along all curves in normal position and without
backtracking.
Therefore, $[\eta]_2\in\AFpar(\triang)$ and
the \Sval\ angle structure given by
$\omega(\square) := (-1)^{\eta(\square)}$
is an element of the component~$C$, which establishes part~\eqref{item:special-points-are-abundant}.

In order to prove part~\eqref{item:Phis-welldef-on-components}, we first consider
two elements $\alpha_1, \alpha_2\in \AFpar(\triang)$
such that $\omega_1 = (-1)^{\alpha_1}$ and $\omega_2=(-1)^{\alpha_2}$
lie in the same connected component of $\SAS_0(\triang)$.
As before, there exist two integer vectors $\eta_1, \eta_2\in\ZZ^{Q(\triang)}$ satisfying
equations 1--4 of Theorem~\ref{generalization-of-6.1} with the same $c$ and $d$
and such that $\alpha_j = [\eta_j]_2$ ($j=1,2$).
Hence, $\eta_1-\eta_2\in\mathbf{L}^\transp(\ZZ^{\edges(\triang)})\subset\TAS_0(\triang)$
and consequently $\alpha_1-\alpha_2\in\Imag\mathbf{L}_2^\transp$.
Applying part~\eqref{item:KerLTD} of Theorem~\ref{homogenous-preimage},
we conclude that $\rect(\alpha_1-\alpha_2)$ is zero in $H^2(M,\partial M;\FF_2)$,
so that $\mathcal{R}_0(\alpha_1)=\mathcal{R}_0(\alpha_2)$.

To finish the proof, we must also consider the case of two
elements $\alpha_1, \alpha_2\in\AFpar(\triang)$
for which $\omega_1 = (-1)^{\alpha_1}$ and $\omega_2=(-1)^{\alpha_2}$
lie in the same connected component of $\SAS(\triang)$, but not necessarily
in the same connected component of $\SAS_0(\triang)$.
In this case, there exist two integer vectors $\eta_1, \eta_2\in\ZZ^{Q(\triang)}$
satisfying equations 1, 2, and 4 of Theorem~\ref{generalization-of-6.1} with the
same $c$ and such that $\alpha_j = [\eta_j]_2$ ($j=1,2$).
Although it may now happen that $\mathbf{G}_\partial \eta_1 \neq \mathbf{G}_\partial \eta_2$,
we still have $\eta_1-\eta_2\in\Imag_\ZZ\mathbf{L}^\transp+\Imag_\ZZ\mathbf{L}_\partial^\transp$.
Recall that $\Imag_\ZZ\mathbf{L}_\partial^\transp$ is spanned over $\ZZ$ by
leading-trailing deformations along peripheral curves.
Thus, using the first part of the proof, we see that the statement will follow if
we show that for any peripheral curve~$\theta$, $\rect([l_\theta]_2)$ represents
the trivial class in $H^2(M;\FF_2)$.
To this end, consider the following commutative diagram
\begin{equation}\label{Xi-diagram}
\begin{tikzcd}
& \Ker\mathbf{G}_2 \arrow[swap]{d}{\rect}\arrow{dr}{\rect}\\
& Z^2(\triang_{00},\partial\triang_{00};\FF_2) \arrow{d}{}\arrow{r}{\subset} & Z^2(\triang_{00};\FF_2)\arrow{d}{}\\
H^1(\partial M;\FF_2)\arrow{r}{\Delta^1}
& H^2(M,\partial M; \FF_2)\arrow{r}{\Xi} & H^2(M;\FF_2) \arrow{r}{\iota^*}
& H^2(\partial M; \FF_2)
\end{tikzcd}
\end{equation}
in which the bottom row is a part of the long exact sequence of the pair
$(M,\partial M)$ in \nword{\FF_2}{cohomology}.
Denote by $\PD_1: H_1(\partial M;\FF_2)\to H^1(\partial M;\FF_2)$ the
Poincar\'e duality isomorphism.
By Lemma~\ref{rectal-lemma},
\[
    [\rect([l_\theta]_2)] = \Delta^1\bigl(\PD_1([\theta])\bigr) \in \Imag\Delta^1=\Ker\Xi,
\]
so that $\rect([l_\theta]_2)$ represents the trivial class in $H^2(M;\FF_2)$.
\end{proof}

\subsection{Proof of Theorem~\ref{obstruction-theorem}}
\label{proof-of-obstruction-theorem}
As shown in Proposition~\ref{Phis-welldef}, the maps $\Phi$ and $\Phi_0$ are well defined
and automatically satisfy part~\eqref{item:continuity} of Theorem~\ref{obstruction-theorem}.
In this section, we are going to focus on the much less obvious parts~\eqref{obstructions-bijectivity}
and~\eqref{item:correct-obstructions}
which follow from the three lemmas stated below.

\begin{lemma}\label{surjectivity-lemma}
    \begin{enumerate}[(i)]
        \item The map $\Phi_0$ is surjective.
        \item The image of $\Phi$ coincides with the kernel of the natural
        map
        \(
            \iota^* : H^2(M;\FF_2)\to H^2(\partial M;\FF_2)
        \).
    \end{enumerate}
\end{lemma}
\begin{lemma}\label{injectivity-lemma}
The maps $\pi_0(\Phi)$ and $\pi_0(\Phi_0)$ are injective.
\end{lemma}
\begin{lemma}\label{obstruction-lemma}
Let $z$ be an algebraic solution of Thurston's edge consistency  equations
defining a conjugacy class of a representation $\rho:\pi_1(M)\to\PSLC$.
Then the associated circle-valued angle
structure~$\omega(\square)=z(\square)/|z(\square)|$ satisfies the following properties:
\begin{enumerate}[(i)]
    \item $\Phi(\omega)\in H^2(M;\FF_2)$ is the cohomological obstruction to
    lifting $\rho$ to an $\SLC$ representation of $\pi_1(M)$.
    \item If $\omega$ has trivial peripheral angle-holonomy, i.e., $\rho$
    is boundary-parabolic, then $\Phi_0(\omega)\in H^2(M,\partial M;\FF_2)$
    is the cohomological obstruction to lifting $\rho$
    to a boundary-unipotent $\SLC$ representation.
\end{enumerate}
\end{lemma}

It is clear that Lemmas~\ref{surjectivity-lemma}--\ref{injectivity-lemma}
jointly imply part~\eqref{obstructions-bijectivity} of Theorem~\ref{obstruction-theorem},
whereas Lemma~\ref{obstruction-lemma} is merely a restatement of part~\eqref{item:correct-obstructions}.
The remainder of this section is devoted to the proof of Lemmas~\ref{surjectivity-lemma}
and~\ref{injectivity-lemma}, whereas the proof of Lemma~\ref{obstruction-lemma}
can be found in Section~\ref{sec:cocycles}.

\begin{proof}[Proof of Lemma~\ref{surjectivity-lemma}]
We start by proving the surjectivity of $\Phi_0$.
Choose any angle structure $\alpha_0\in\AFpar(\triang)$ and
let $\nu\in H^2(M,\partial M;\FF_2)$ be an arbitrary cohomology class.
Applying part~(i) of Theorem~\ref{homogenous-preimage} to the class
$\nu' = \mathcal{R}_0(\alpha_0)-\nu$ yields an element
$s\in\mathcal{S}$ such that $\nu' = [\rect(s)]$.
It is easy to see that $\alpha_0 - s \in\AFpar(\triang)$, because $\mathcal{S}$
is the vector space associated to the affine space~$\AFpar(\triang)\subset\FF_2^{Q(\triang)}$.
Moreover,
\[
    \mathcal{R}_0(\alpha_0-s) = [\rect(\alpha_0) - \rect(s)]
    = (\nu' + \nu) - \nu' = \nu;
\]
since $\nu$ was arbitrary, $\Phi_0$ is surjective.

In order to prove part~(ii), consider the exponential map
\begin{equation}\label{sign-exponential}
    \epsilon: \AFpar(\triang)\to\SAS_0(\triang),
    \qquad
    \epsilon(\alpha)(\square) = (-1)^{\alpha(\square)}.
\end{equation}
Using the commutative diagram~\eqref{Xi-diagram}, we see that
the equality $\Phi = \Xi\circ\Phi_0$ holds
on $\epsilon\bigl(\AFpar(\triang)\bigr)$.
Since we have shown above that
$(\Phi_0\circ\epsilon)\bigl(\AFpar(\triang)\bigr)=H^2(M,\partial M;\FF_2)$,
we must have $\Imag\Phi = \Imag\Xi = \Ker\iota^*$.
\end{proof}

\begin{proof}[Proof of Lemma~\ref{injectivity-lemma}]
In order to prove that $\pi_0(\Phi_0)$ is injective,
consider two elements $\alpha_1,\alpha_2\in\AFpar(\triang)$
such that $\Phi_0\bigl((-1)^{\alpha_1}\bigr)=\Phi_0\bigl((-1)^{\alpha_2}\bigr)$.
Set $s=\alpha_1-\alpha_2$, so that $\rect(s)$ represents the trivial class
in $H^2(M,\partial M; \FF_2)$.
Since $s\in\mathcal{S}$, part~(ii) of Theorem~\ref{homogenous-preimage} implies that
$s\in \Imag\mathbf{L}_2^\transp$.
In other words, there exists an integer vector $r\in\ZZ^{\edges(\triang)}$ such that
$s = \bigl[\mathbf{L}^\transp r\bigr]_2$.
By Remark~\ref{TAS-and-connected-components},
the circle-valued angle structures~$(-1)^{\alpha_1}$ and $(-1)^{\alpha_2}$ lie
in the same connected component of $\SAS_0(\triang)$.

It remains to be shown that $\pi_0(\Phi)$ is injective as well.
To this end, suppose that $\alpha_1,\alpha_2\in\AFpar(\triang)$
satisfy $\Phi\bigl((-1)^{\alpha_1}\bigr)=\Phi\bigl((-1)^{\alpha_2}\bigr)$.
Let $s=\alpha_1-\alpha_2$, so that $\rect(s)$ represents
the trivial class in $H^2(M;\FF_2)$.
We must show that $\alpha_1$ and $\alpha_2$ lie in the same connected
component of $\SAS(\triang)$.
In view of Remark~\ref{TAS-and-connected-components}, this is equivalent to showing
that $s$ is an element of $\Imag\mathbf{L}_2^\transp + \Imag[\mathbf{L_\partial}]_2^\transp$.

Referring to the commutative diagram~\eqref{Xi-diagram}, we may first
consider the \emph{relative} cohomology class~$[\rect(s)]\in H^2(M,\partial M;\FF_2)$
which satisfies $[\rect(s)]\in\Ker\Xi=\Imag\Delta^1$.
Let $\Theta$ be a finite set of peripheral curves in normal position with respect to $\triang$
whose homology classes span $H_1(\partial M;\FF_2)$; it suffices to take two curves in
each boundary component.
Then
\[
    \Imag\Delta^1 = \linspan_{\FF_2} \left\{\Delta^1\bigl(\PD_1([\theta])\bigr):\theta\in\Theta\right\},
\]
where $\PD_1:H_1(\partial M;\FF_2)\to H^1(\partial M;\FF_2)$ is the Poincar\'e duality isomorphism.
By Lemma~\ref{rectal-lemma}, $[\rect(s)]\in\Imag\Delta^1$ must therefore
be a linear combination of relative cohomology classes of the form $[\rect([l_\theta]_2)]$
for $\theta\in\Theta$, i.e., $[\rect(s)]=[\rect(x)]\in H^2(M,\partial M;\FF_2)$
for some element $x\in\Imag[\mathbf{L_\partial}]_2^\transp$.
By Remark~\ref{ltds-mod-2-contained-in-S}, we have $\Imag[\mathbf{L_\partial}]_2^\transp\subset\mathcal{S}$,
so that $s-x\in\mathcal{S}$.
Since $\rect(s-x)$ is zero in $H^2(M,\partial M;\FF_2)$,
part~(ii) of Theorem~\ref{homogenous-preimage} implies that $s-x\in\Imag\mathbf{L}_2^\transp$,
so that $s\in\Imag\mathbf{L}_2^\transp + \Imag[\mathbf{L_\partial}]_2^\transp$, as required.
\end{proof}

%==============================================================================
% Section 11: Explicit obstruction cocycles
%==============================================================================

\section{Obstruction cocycles from hyperbolic shapes}
\label{sec:cocycles}

Suppose that $\rho=\rho_z:\pi_1(M)\to\PSLC$ is a representation determined, up
to conjugacy, by an algebraic solution~$z$ of Thurston's gluing equations on an oriented ideal triangulation 
$\triang$.  The goal of this section is to construct an explicit cellular
cocycle representing the obstruction class to lifting $\rho$ to an $\SLC$
representation. This cocycle will be defined on the CW complex of doubly-truncated tetrahedra $\triang_{00}$ used in the previous section,
The construction will be carried out in terms of the solution
$z$ as well as an extra piece of structure introduced by W.~Neumann, called a
\emph{strong combinatorial flattening} \cite{neumann-BPSL}.

The main theorem of this section is the
following.

\begin{theorem}\label{rect=Obs}
Let $z$ be an algebraic solution of Thurston's edge consistency equations on $\triang$ and
let $\rho=\rho_z$ be an associated representation $\rho:\pi_1(M)\to\PSLC$.
\begin{enumerate}[(i)]
\item\label{item:noperiph-Obs}
The connected component of $\SAS(\triang)$ containing
the circle-valued angle structure~$\omega(\square)=z(\square)/|z(\square)|$ also contains an
element of the form $(-1)^{\alpha}$, where $\alpha\in\AFpar(\triang)$
is such that $\Obs(\rho) = [\rect(\alpha)]\in H^2(M;\FF_2)$.
\item\label{item:periph-Obs}
If, additionally, $z$ satisfies the completeness equations, so that $\rho$ is boundary-parabolic,
then the connected component of $\SAS_0(\triang)$ containing
$\omega$ also contains an element of the form $(-1)^{\alpha}$, where $\alpha\in\AFpar(\triang)$
is such that $\Obs_0(\rho) = [\rect(\alpha)]\in H^2(M,\partial M;\FF_2)$.
\end{enumerate}
\end{theorem}
In particular, the above theorem and Definition~\ref{def-Phi}
immediately imply Lemma~\ref{obstruction-lemma}.

The obstruction cocycles we will be studying are discussed by \cite{culler-lifting,acuna-montesinos}, and by \cite{zickert-volrep} for the case with peripheral constraints.

\subsection{Obstruction theory via group cocycles}

In what follows, we fix an algebraic solution $z\in\left(\CC\setminus\{0,1\}\right)^{Q(\triang)}$
of the edge consistency and completeness equations
on $\triang$ and we suppose that $\rho=\rho_z:\pi_1(M)\to\PSLC$ is the associated boundary-parabolic representation,
defined up to conjugation only.
An explicit way of describing the conjugacy class of $\rho$ is via a
\nword{G}{cocycle} on a cell complex representing $M$, where $G=\PSLC$.
A particularly simple construction of this kind, in terms of the shape
parameters $z$ and the cell complex $\triang_{00}$, is described by
the third named author in \cite[\S2.3.2]{siejakowski-thesis}.

Since we are interested in lifting $\rho$ to $\SLC$,
we are going to label the oriented \nword{1}{cells} of the cell complex $\triang_{00}$
with elements of $\SLC$ which project to a \nword{\PSLC}{cocycle} under
the quotient homomorphism $\SLC\to\PSLC$.
The failure of such labellings to satisfy the full cocycle conditions
over $\SLC$ will then lead to an explicit obstruction cocycle for $\rho$.
However, in order to find the \emph{relative} obstruction
class~$\Obs_0(\rho)\in H^2(M,\partial M;\FF_2)$ in this way,
we must choose an \nword{\SLC}{labelling} which already satisfies the full
cocycle conditions on $\partial M$, so that the associated obstruction class vanishes
on the boundary.
This is always possible, because the boundary subgroups are abelian.

\begin{definition}[boundary-compatible labelling for $z$]\label{def:nice-cocycle}
Let $\mathcal{L}$ be a labelling of oriented \nword{1}{cells} of $\triang_{00}$
with elements of $\SLC$.
Then, $\mathcal{L}$ is a \emph{boundary-compatible labelling for $z$} if it satisfies the
following conditions.
\begin{enumerate}[(i)]
    \item\label{item:lab-SL-cocyc-bdry}
    The restriction of $\mathcal{L}$ to $\partial\triang_{00}$ is an \nword{\SLC}{cocycle}.
    \item\label{item:unipotent}
    For any component $T$ of $\partial\triang_{00}$, the restriction $\mathcal{L}\bigr|_T$
    determines a unipotent representation of $\pi_1(T)$ (up to conjugation).
    \item\label{item:lab-desc-to-PSL}
    Under the quotient homomorphism $\SLC\to\PSLC$, the labelling~$\mathcal{L}$
    descends to a \nword{\PSLC}{cocycle} recovering the conjugacy class of
    the boundary-parabolic representation~$\rho$ determined by the shapes $z$.
\end{enumerate}
\end{definition}

Given a boundary-compatible labelling~$\mathcal{L}$, we can easily find
a cellular representative of $\Obs_0(\rho)$.
To this end, we construct a cellular \nword{2}{cocycle} whose value on an oriented \nword{2}{cell}
is equal to the product of the labels around the oriented boundary of the cell;
this product will always be $+1$ or $-1$, where $1\in\SLC$ is the identity matrix.
By identifying the multiplicative group $\{\pm1\}$ with the additive group $\FF_2$
of the two-element field, we obtain an \nword{\FF_2}{valued} relative cocycle
representing the obstruction class~$\Obs_0(\rho)\in H^2(M,\partial M;\FF_2)$.

\subsection{Combinatorial flattenings}
\label{combinatorial_flattenings_section}

Constructing a boundary-compatible labelling for any given shape parameter
solution~$z$ of Thurston's edge consistency and completeness equations
will require an auxiliary object, called a \emph{strong flattening for $z$},
introduced by W.~Neumann~\cite{neumann-BPSL}.

Since the shape parameters $z$ satisfy $z(\square)z(\square')z(\square'')=-1$,
we can always choose branches of their logarithms $Z(\square)=\log z(\square)$
in such a way that $Z(\square)+Z(\square')+Z(\square'')=i\pi$ in every
tetrahedron.
These log-parameters will then satisfy gluing equations of the form
\begin{equation}\label{lifted-logeqs}
\left\{
    \begin{aligned}
        \mathbf{G} Z &= \pi i c,\\
        \mathbf{G}_\del Z &= \pi i d,
    \end{aligned}\right.
\end{equation}
where $c\in2\ZZ^{\edges(\triang)}$ and where
\(
    d = (d^1_\mu, d^1_\lambda, d^2_\mu, d^2_\lambda, \dotsc, d^k_\mu, d^k_\lambda)^\transp
\)
is a vector with even entries corresponding to the chosen
system~$\{\mu_j,\lambda_j\}_{j=1}^k$ of peripheral curves satisfying $\intersection(\mu_j,\lambda_j)=1$;
cf. also eq.~\eqref{completeness-equations-mod-2pi}.
In \cite{neumann-BPSL}, Neumann introduced a way of adjusting the log-parameters~$Z$
by integer multiples of $\pi i$ so as to homogenise the equations \eqref{lifted-logeqs}.
\begin{definition}[cf. Neumann~{\cite[Definition~3.1]{neumann-BPSL}}]\label{def:flattening}
A \emph{combinatorial flattening} for $(\triang,z)$ is a vector $W\in\CC^{Q(\triang)}$
of the form
\begin{equation}\label{form-of-flattening}
    W(\square) = Z(\square) - i\pi f(\square),\quad f(\square)\in\ZZ
\end{equation}
satisfying the following conditions:
\begin{enumerate}[(i)]
\item $\displaystyle W(\square)+W(\square')+W(\square'') = 0$ in every tetrahedron;
\item\label{item:sum-of-flatt-around-edge-equals-zero}
$\displaystyle \mathbf{G}W=0$,
\item $\displaystyle \mathbf{G}_\partial W=0$, i.e., the signed sum along all peripheral curves is zero.
\end{enumerate}
\end{definition}

Note that for any combinatorial flattening~$W$, the parity of the vector
$f\in\ZZ^{Q(\triang)}$ occurring in \eqref{form-of-flattening} does not depend on the choice of branches
of the logarithms.
In other words, we can associate to $W$ a well-defined vector $\delta=\delta_W\in\FF_2^{Q(\triang)}$
given by the remainder of $f$ modulo 2:
\begin{equation}\label{def-parity-of-flattening}
    \delta(\square) = [f(\square)]_2\in\FF_2\quad\text{for all}\ \square\in Q(\triang).
\end{equation}

\begin{definition}\label{strong-flattening}
A combinatorial flattening~$W$ is called \emph{strong} when $\parity_\theta(\delta)=0$ for
all unoriented curves~$\theta$ in normal position with respect to the triangulation
and without backtracking, in the sense of Definition~\ref{def-parity}.
\end{definition}

\begin{theorem}[Neumann,~{\cite[{Theorem~4.5}]{neumann-BPSL}}]\label{th:strong-flatt-existence}
Given any algebraic solution~$z$ of edge consistency and completeness equations,
there exists a strong combinatorial flattening~$W$ for $z$.
\end{theorem}
Neumann's proof of the above theorem in \cite{neumann-BPSL} relies on the combinatorial
results of \cite{neumann1990}. A simple proof also follows from Theorem~\ref{generalization-of-6.1}, proved in this article.

\subsection{From flattenings to projective cocycles}

From now on, we assume that $W$ is a strong combinatorial flattening for the
algebraic solution~$z$ of edge consistency and completeness equations on $\triang$.
Using $W$, we will produce an explicit boundary-compatible labelling for $z$ in the sense of
Definition~\ref{def:nice-cocycle}.

\subsubsection{\texorpdfstring{Choice of orientations of 1-cells in $\mathcal{T}_{00}$}%
{Choice of orientations of 1-cells in T\char`_00}}

Let $F_0$ be an arbitrarily chosen face of the triangulation $\triang$,
equipped with a transverse orientation (co-orientation).
With the help of the parity vector $\delta$ defined in \eqref{def-parity-of-flattening},
we can unambiguously extend this co-orientation to all other faces of $\triang$ according to the following rule.
\begin{ruledup}\label{coorient-rule}
Suppose that $F_1$ and $F_2$ are two faces of a tetrahedron with a parity
parameter $\delta(\square)\in\FF_2$ along the tetrahedral edge common to $F_1$ and $F_2$.
\begin{itemize}
\item If $\delta(\square)=0$, one of the two faces must be co-oriented into the tetrahedron and the other
out of the tetrahedron;
\item If $\delta(\square)=1$, either both faces must be co-oriented into the tetrahedron or both out of the tetrahedron.
\end{itemize}
\end{ruledup}
Observe that the three parities in a tetrahedron always satisfy
$\delta(\square)+\delta(\square')+\delta(\square'')=1$, so that either exactly one of them is $1$ or all
three of them are $1$.
In the latter case, either all faces of the tetrahedron are co-oriented inwards or all outwards.%
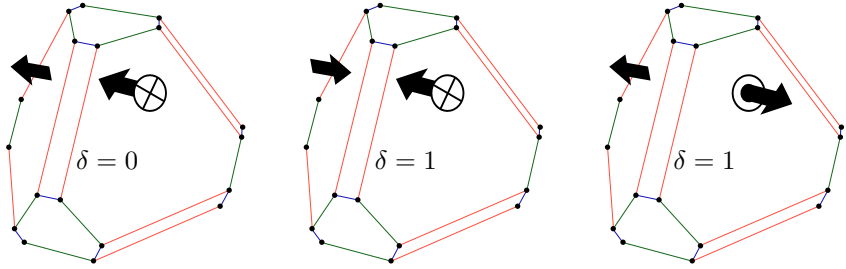
\begin{figure}
     \centering
     \begin{tikzpicture}
         \node[anchor=south west,inner sep=0] (image) at (0,0)
         {\includegraphics[width=0.75\columnwidth]{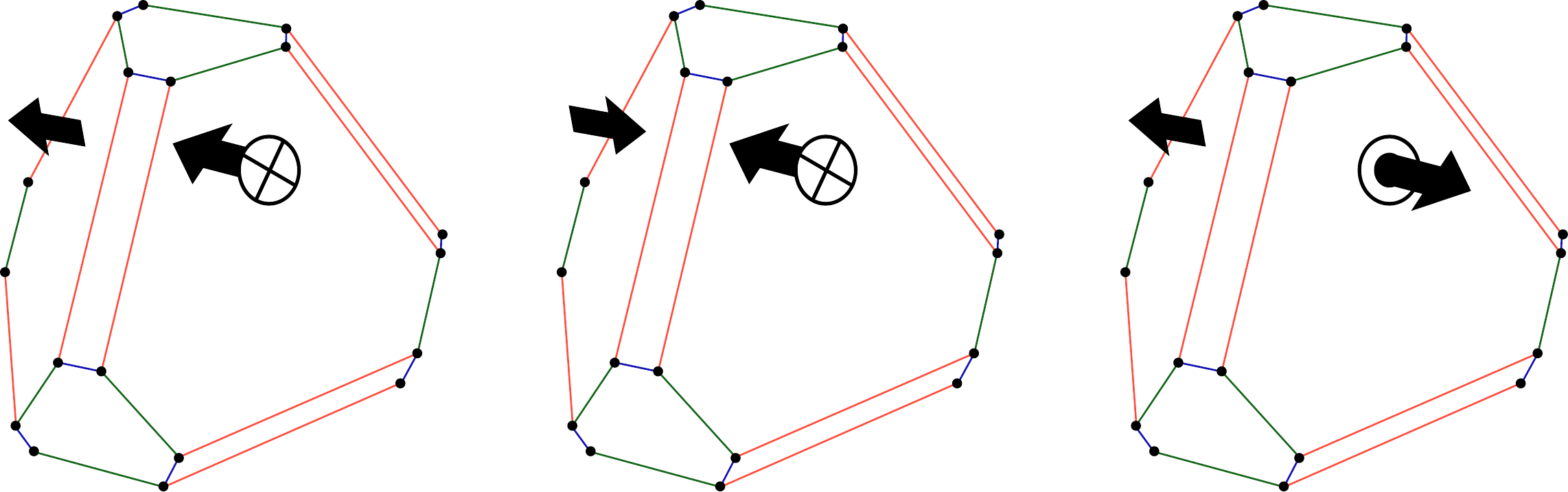}};
         \begin{scope}[x={(image.south east)},y={(image.north west)}]
             \node[anchor=west] at (0.073, 0.4)  {$\delta=0$};
             \node[anchor=west] at (0.429, 0.4)  {$\delta=1$};
             \node[anchor=west] at (0.785, 0.4)  {$\delta=1$};
         \end{scope}
     \end{tikzpicture}
     \caption{%
     The rule for continuing the co-orientations of faces across a tetrahedron
     depending on the value of the parity vector $\delta$ on the normal quadrilateral
     facing the edge common to the two faces.
     }\label{fig:coorient-monodromy}
\end{figure}
Rule~\ref{coorient-rule} is illustrated in Figure~\ref{fig:coorient-monodromy}.

In order to see that Rule~\ref{coorient-rule} determines a consistent co-orientation of all faces
of $\triang$, it suffices to consider a closed path $\gamma$ based at the initial face $F_0$.
We may assume that $\gamma$ lies in the complement of the edges of $\triang$,
that it intersects the faces transversely, and that it does not backtrack.
Since the combinatorial flattening $W$ is strong,
we have $\parity_\gamma(\delta)=0$, and the co-orientations of the large hexagonal
faces intersected by $\gamma$, when continued along $\gamma$ in accordance with Rule~\ref{coorient-rule},
will switch an even number of times.
In particular, the continued co-orientation of $F_0$ will agree with the original one.
In fact, the co-orientation of the initial face $F_0$ is the only arbitrary choice in this construction.

\begin{remark}
An alternative version of the above statement is in terms of an \nword{\FF_2}{cocycle}
which assigns $1$ to all long and medium edges and $\delta(\square)$ to each short edge facing the quad~$\square$.
This cocycle determines an \nword{\FF_2}{cover} of $M$ given as the total
space of the principal \nword{\FF_2} bundle corresponding to the cocycle.
The vanishing of the parity of $\delta$ implies that this cover is trivial, i.e., homeomorphic to $M\times\FF_2$,
and the choice of the co-orientation of the initial face~$F_0$ amounts to picking one of the
two sheets of the cover.
\end{remark}

We shall now use these consistent co-orientations of the large hexagonal faces of $\triang_{00}$
to orient the \nword{1}{cells} of $\triang_{00}$.
To this end, we fix an arbitrary orientation of the manifold $M$.
Recall that there are three types of edges in $\triang_{00}$: the \emph{long},
\emph{medium} and \emph{short} edges, shown in Figure~\ref{fig:doubly-truncated}.

Suppose that the initial co-orientation of $F_0$ has been continued onto all faces of the triangulation
$\triang$, hence onto all large hexagonal faces of $\triang_{00}$.
We may now use the right-hand rule to canonically orient the long and medium edges in $\triang_{00}$,
as illustrated in Figure~\ref{fig:coorientation}.
\begin{figure}
    \centering
    \parbox{0.25\columnwidth}{%
    \includegraphics[width=0.25\columnwidth]{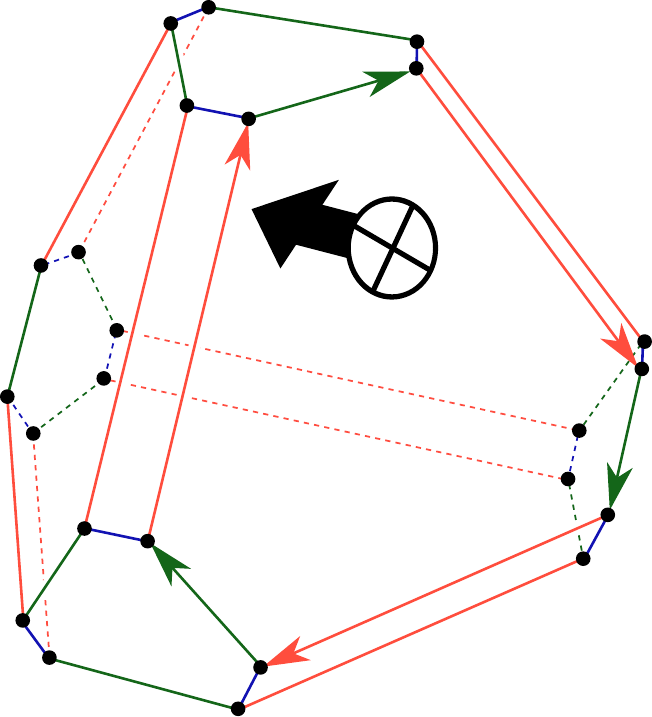}%
    }\parbox{0.5\columnwidth}{
    \caption{\label{fig:coorientation}
    The rule for orienting the long and medium edges based on a co-orientation of a
    large hexagonal face.
    }}
\end{figure}
It remains to orient the short edges of $\triang_{00}$, which is done with the following rule.
\begin{ruledup}\label{rule:short-edges}
We use an outward facing normal vector at each end of an edge of $\triang$ and orient the
short edges using the right-hand rule (hence, anticlockwise when looking from the boundary).
\end{ruledup}

Although we have chosen an arbitrary ambient orientation of $M$, choosing the opposite orientation
results in the reversal of all edges in $\triang_{00}$, which does not affect the considerations that follow.

\subsubsection{\texorpdfstring{The labelling with $\SLC$ matrices}{The labelling with SL(2,C) matrices}}
Suppose that the \nword{1}{cells} of the complex~$\triang_{00}$ are all oriented
in the way described above, in terms of the parity $\delta=\delta_W$ of a strong combinatorial
flattening $W$.
We may now label these oriented \nword{1}{cells} with the elements of $\SLC$ given below.

\begin{ruledup}[Labelling rule]\hfill\\[-4ex]\label{labelling-rule}
\begin{itemize}
\item Each oriented long edge is labelled by the matrix
    $\displaystyle S = \begin{bmatrix}0 & -1\\1 & 0\end{bmatrix}$.
\item Each oriented medium edge is labelled by the matrix
    $\displaystyle T = \begin{bmatrix}1 & -1\\0 & 1\end{bmatrix}$.
\item Each oriented short edge is labelled by the matrix
    \[
    H_{W(\square)}
    = \begin{bmatrix} \exp\bigl(\frac12 W(\square)\bigr) & 0\\
    0 & \exp\bigl(-\frac12 W(\square)\bigr)\end{bmatrix},
    \]
    where $W(\square)$ is the value of the strong combinatorial flattening~$W$ on the quad
    parallel to the rectangle whose boundary contains the short edge.
\end{itemize}
\end{ruledup}
Since we wish for this labelling to descend to a projective cocycle under the
homomorphism $\SLC\to\PSLC$, it is to be understood that changing the orientation of an edge
replaces its label~$A\in\SLC$ with $A^{-1}$.
Observe in particular that $\displaystyle T^{-1} = \begin{bmatrix}1 & +1\\0 & 1\end{bmatrix}$
and $S^{-1}=-S$.

\subsection{Cocycle conditions}
Our main goal now is to prove that the labelling $\mathcal{L}$ defined
by Rule~\ref{labelling-rule} has all the desired properties needed
to compute obstruction cocycles.

\begin{lemma}\label{labelling-works}
For every strong combinatorial flattening $W$ for $z$,
the labelling $\mathcal{L}$ of Rule~\ref{labelling-rule} is a boundary-compatible
labelling for $z$ in the sense of Definition~\ref{def:nice-cocycle}.
\end{lemma}

\begin{proof}
We split the proof into several steps, verifying one by one the conditions
in Definition~\ref{def:nice-cocycle}.
\Proofpart{Step 1:}{$\mathcal{L}|_{\partial M}$ is an \nword{\SLC}{cocycle}.}
In order to check condition \eqref{item:lab-SL-cocyc-bdry}, we must verify the
\nword{\SLC}{cocycle} conditions for all boundary \nword{2}{cells} of $\triang_{00}$.
We start by considering a boundary hexagon coming from a truncated vertex of a tetrahedron.
There are two cases: either all parities
$\delta(\square)$, $\delta(\square')$, $\delta(\square'')$
in the tetrahedron are odd or exactly one of them is odd.
We shall write $\delta$, $\delta'$, $\delta''$ for short and consider both cases separately.

\Proofpart{Case 1.1:}{$\delta=\delta'=\delta''=1$.}
Consider a tetrahedron in which the parities are all odd.
Hence, the face co-orientations are either all inwards or all outwards.
Figure~\ref{fig:horohex-all-out} shows the case of all co-orientations pointing outwards.
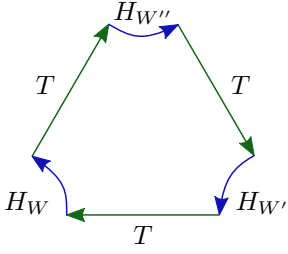
\begin{figure}
    \centering
    \parbox{0.25\columnwidth}{%
    \begin{tikzpicture}
        \node[anchor=south west,inner sep=0] (image) at (0,0)
            {\includegraphics[width=0.2\columnwidth]{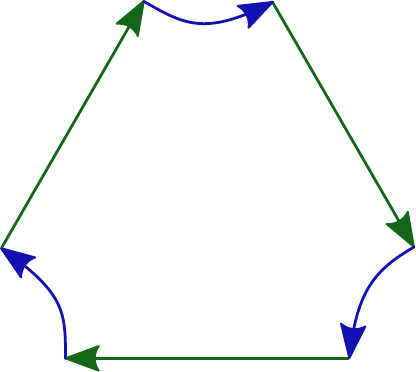}};
        \begin{scope}[x={(image.south east)},y={(image.north west)}]
            \node[anchor=north east] at (0.13, 0.2)  {$H_W$};
            \node[anchor=north west] at (0.87, 0.2)  {$H_{W'}$};
            \node[anchor=south] at (0.5, 0.95)  {$H_{W''}$};
            \node[anchor=south east] at (0.15, 0.6) {$T$};
            \node[anchor=south west] at (0.85, 0.6) {$T$};
            \node[anchor=north] at (0.5,0.03) {$T$};
        \end{scope}
    \end{tikzpicture}%
    }\parbox{0.72\columnwidth}{
    \caption{%
    \label{fig:horohex-all-out}
    A boundary hexagon in a doubly truncated tetrahedron with $\delta=\delta'=\delta''=1$,
    viewed from the ideal vertex, with a possible orientation of its edges.
    }}
\end{figure}
The holonomy around the boundary hexagon is then
\begin{align*}
    H_{W''} T H_{W'} T H_W T
    &=
    \begin{bmatrix}e^{W''/2} & -e^{W''/2}\\ 0 & e^{-W''/2}\end{bmatrix}
    \begin{bmatrix}e^{W'/2} & -e^{W'/2}\\ 0 & e^{-W'/2}\end{bmatrix}
    \begin{bmatrix}e^{W/2} & -e^{W/2}\\ 0 & e^{-W/2}\end{bmatrix}
    \\
    &=
    \begin{bmatrix}e^{(W+W'+W'')/2} & -b\\ 0 & e^{-(W+W'+W'')/2}\end{bmatrix}\\
    &= \begin{bmatrix}1 & -b\\ 0 & 1\end{bmatrix},
\end{align*}
where
\begin{align*}
    b &= e^{(W+W'+W'')/2}+e^{(-W+W'+W'')/2}+e^{(-W-W'+W'')/2}\\
    &= 1 + e^{-W} + e^{W''}
    = 1 + e^{-Z+i\pi f} + e^{Z''-i\pi f''}\\
    &= 1 + (-1)^\delta\tfrac{1}{z} + (-1)^{\delta''}z''
    = 1-\tfrac{1}{z} - \bigl(1-\tfrac{1}{z}\bigr)
    = 0.
\end{align*}
When all faces of the tetrahedron are co-oriented inwards, so that
the medium edges have orientations opposite to those shown in
Figure~\ref{fig:horohex-all-out}, the holonomy around the hexagon is
\[
    H_W^{-1} T H_{W'}^{-1} T H_{W''}^{-1} T = H_{-W}TH_{-W'}TH_{-W''}T.
\]
Using the previous calculation, we find the above element to be equal to
$\left[\begin{smallmatrix}1&-b'\\0&1\end{smallmatrix}\right]$, where this time
\begin{align*}
    b' &= e^{-(W+W'+W'')/2}+e^{(-W-W'+W'')/2}+e^{(-W+W'+W'')/2}\\
     &= 1 + e^{W''} + e^{-W} = b = 0.
\end{align*}
Therefore, the $\SLC$ cocycle condition holds around the boundary hexagon.

\Proofpart{Case 1.2:}{$\delta=1$, $\delta'=\delta''=0$.}
This is the case of exactly one of the three parity parameters being odd;
without loss of generality, we take it to be $\delta$.
An example of this situation is shown in Figure~\ref{fig:horohex-taut}.
\begin{figure}
    \centering
    \parbox{0.25\columnwidth}{%
    \begin{tikzpicture}
        \node[anchor=south west,inner sep=0] (image) at (0,0)
            {\includegraphics[width=0.2\columnwidth]{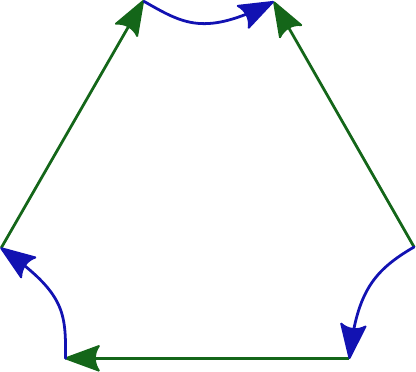}};
        \begin{scope}[x={(image.south east)},y={(image.north west)}]
            \node[anchor=north east] at (0.13, 0.2)  {$H_W$};
            \node[anchor=north west] at (0.87, 0.2)  {$H_{W'}$};
            \node[anchor=south] at (0.5, 0.95)  {$H_{W''}$};
            \node[anchor=south east] at (0.15, 0.6) {$T$};
            \node[anchor=south west] at (0.85, 0.6) {$T$};
            \node[anchor=north] at (0.5,0.03) {$T$};
        \end{scope}
    \end{tikzpicture}%
    }\parbox{0.73\columnwidth}{
    \caption{%
    A possible orientation of edges around a boundary hexagon
    in a tetrahedron with $\delta=1$ and $\delta'=\delta''=0$.
    }\label{fig:horohex-taut}}
\end{figure}
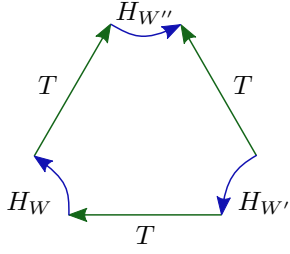
In the case presented in the figure, the holonomy
around the boundary hexagon equals:
\begin{align*}
    H_{W''} T^{-1} H_{W'} T H_W T
    &=
    \begin{bmatrix}e^{W''/2} & e^{W''/2}\\ 0 & e^{-W''/2}\end{bmatrix}
    \begin{bmatrix}e^{W'/2} & -e^{W'/2}\\ 0 & e^{-W'/2}\end{bmatrix}
    \begin{bmatrix}e^{W/2} & -e^{W/2}\\ 0 & e^{-W/2}\end{bmatrix}
    \\
    &=
    \begin{bmatrix}e^{(W+W'+W'')/2} & -b\\ 0 & e^{-(W+W'+W'')/2}\end{bmatrix}\\
    &= \begin{bmatrix}1 & -b\\ 0 & 1\end{bmatrix},
\end{align*}
where
\begin{align*}
    b &= e^{(W+W'+W'')/2} + e^{(-W+W'+W'')/2} - e^{(-W-W'+W'')/2}\\
    &= 1+e^{-W}-e^{W''}
    = 1+e^{-Z+i\pi f}-e^{Z''-i\pi f''}\\
    &= 1+(-1)^{\delta}\tfrac{1}{z} - (-1)^{\delta''}z''
    = 1-\tfrac{1}{z} - \bigl(1-\tfrac{1}{z}\bigr)
    = 0.
\end{align*}
Likewise, when the directions of all medium arrows are reversed,
we obtain the holonomy
\[
    H_{W'}^{-1}T^{-1}H_{W''}^{-1}TH_W^{-1}T
    = H_{-W'}T^{-1} H_{-W''} T H_{-W} T.
\]
Using the previous calculation, this matrix product
equals $\left[\begin{smallmatrix}1&-b'\\0&1\end{smallmatrix}\right]$,
where
\begin{align*}
    b' &=  e^{-(W+W'+W'')/2} + e^{(W-W'-W'')/2} - e^{(W-W'+W'')/2}\\
    &= 1 + e^W - e^{-W'}
    = 1+(-1)^\delta z - (-1)^{\delta'}\tfrac{1}{z'}\\
    &= 1 - z - (1-z) = 0.
\end{align*}
In conclusion, the $\SLC$ cocycle condition around boundary hexagons holds in all cases.

In order to verify the cocycle condition around the edge discs,
we consider an edge $E$ of $\triang$.
Since the orientations of Rule~\ref{rule:short-edges} are consistent around the boundary
of each of the two discs capping off the cylindrical 3-cell around $E$,
the holonomy along these boundaries is always
\[
    \begin{bmatrix} \exp\bigl(\frac12 \sum_{\square} G(E,\square)W(\square)\bigr) & 0\\
    0 & \exp\bigl(-\frac12 \sum_{\square} G(E,\square)W(\square)\bigr)\end{bmatrix}
    = \begin{bmatrix}1 & 0\\ 0 & 1\end{bmatrix},
\]
where we have used property~\eqref{item:sum-of-flatt-around-edge-equals-zero} of
Definition~\ref{def:flattening}.

\Proofpart{Step 2:}{Boundary unipotence.}
Thanks to Step~1 above, we know that the restriction of the labelling $\mathcal{L}$
to $\partial\triang_{00}$ defines a flat $\SLC$-bundle
on each component of $\partial M$.
We will now show that for each boundary component, the monodromy
of this bundle is always unipotent.

Given a closed peripheral curve $\gamma$,
we can homotope $\gamma$ into the \nword{1}{skeleton} of the cell
complex $\partial\triang_{00}$, so that $\gamma$ is made up of
short and medium edges only.
Therefore, the monodromy along $\gamma$ will be a product of elements of the form
$H_{W(\square)}^{\pm 1}$ and $T^{\pm 1}$.
These matrices are all upper-triangular, so the monodromy along $\gamma$ is also
an upper-triangular matrix whose top-left entry is the product of top-left entries
of all the matrices being multiplied.
Since the top-left entry of $T^{\pm 1}$ is $1$, only the matrices $H_{W(\square)}^{\pm 1}$
contribute to this product.
More precisely, the \nword{(1,1)}{entry} of the monodromy along $\gamma$ equals
\[
     \prod_{\square\in Q(\triang)} \left(e^{W(\square)/2}\right)^{G(\gamma,\square)}
     = \exp\Bigl(\tfrac{1}{2}\sum_{\square\in Q(\triang)} G(\gamma,\square)W(\square)\Bigr) = 1.
\]
Therefore, both diagonal entries of the monodromy matrix are equal to $1$.

\Proofpart{Step 3:}{$\mathcal{L}$ projects to a \nword{\PSLC}{cocycle}.}
We must now show that the elements assigned by $\mathcal{L}$ to oriented
edges of $\triang_{00}$ project to a \nword{\PSLC}{cocycle} under the quotient
homomorphism $\SLC\to\PSLC$.
Step~1 implies that the cocycle condition holds around boundary \nword{2}{cells}
of $\triang_{00}$, so it suffices to study the large hexagonal faces and the rectangles.

As shown in Figure~\ref{fig:coorientation}, the holonomy around the boundary of
a large hexagonal face has the form $(TS)^3$ and is therefore equal to the identity matrix,
as can be verified by direct computation.

It remains to study the holonomy around a rectangle.
We denote by $W=W(\square)$ and $\delta=\delta(\square)$ the quantities
assigned to the normal quad $\square$ parallel to the rectangle and consider two cases
depending on the value of $\delta$.

\Proofpart{Case 3.1:}{$\delta=0$.}
\begin{figure}
    \centering
    \begin{tikzpicture}
        \node[anchor=south west,inner sep=0] (image) at (0,0)
            {\includegraphics[width=0.17\columnwidth]{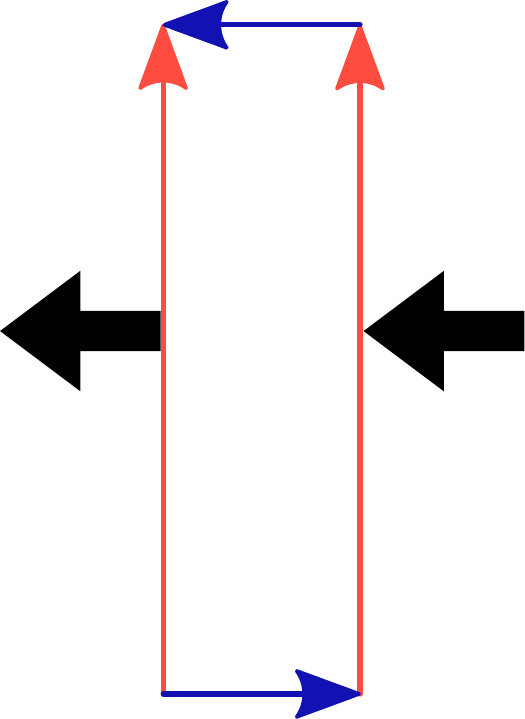}};
        \begin{scope}[x={(image.south east)},y={(image.north west)}]
            \node[anchor=south] at (0.5, 0.97)  {$H_W$};
            \node[anchor=north] at (0.5, 0.03)  {$H_W$};
            \node[anchor=west] at (0.73, 0.9)   {$S$};
            \node[anchor=east] at (0.27, 0.9)   {$S$};
            \node at (0.5, 0.54) {\small $\delta=0$};
        \end{scope}
    \end{tikzpicture}
    \hspace{0.1\columnwidth}
    \begin{tikzpicture}
        \node[anchor=south west,inner sep=0] (image) at (0,0)
            {\includegraphics[width=0.17\columnwidth]{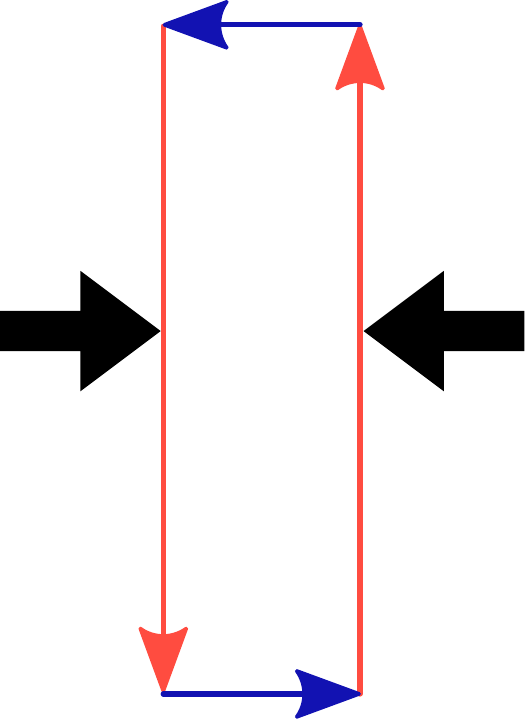}};
        \begin{scope}[x={(image.south east)},y={(image.north west)}]
            \node[anchor=south] at (0.5, 0.97)  {$H_W$};
            \node[anchor=north] at (0.5, 0.03)  {$H_W$};
            \node[anchor=west] at (0.73, 0.9)   {$S$};
            \node[anchor=east] at (0.27, 0.1)   {$S$};
            \node at (0.5, 0.54) {\small $\delta=1$};
        \end{scope}
    \end{tikzpicture}%
    \caption{\label{fig:rects}
    Holonomy around a rectangular \nword{2}{cell} in $\triang_{00}$,
    as viewed from outside the tetrahedron.
    Shown are two cases depending on the value of the parity parameter $\delta$
    on the corresponding normal quadrilateral type.
    The co-orientations of the adjacent large hexagonal faces, represented
    by the thick black arrows, induce the orientations of the long edges
    as in Figure~\ref{fig:coorientation}.
    }
\end{figure}
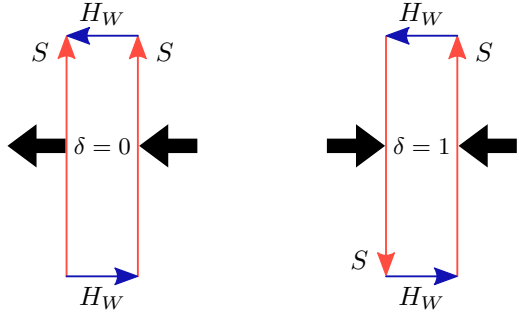
This case is illustrated on the left side of Figure~\ref{fig:rects}.
The holonomy around the rectangle is
\[
    H_W S H_W S^{-1} = - \left(H_W S\right)^2 = \begin{bmatrix}
    1 & 0\\
    0 & 1
    \end{bmatrix}.
\]
In other words, the $\SLC$ cocycle condition is satisfied on a rectangle with $\delta=0$,
which in particular implies the $\PSLC$ cocycle condition.

\Proofpart{Case 3.2:}{$\delta=1$.}
This case is illustrated on the right side of Figure~\ref{fig:rects}.
The holonomy around the rectangle is
\[
    H_W S H_W S = \left(H_W S\right)^2 =
    -\begin{bmatrix}
    1 & 0\\
    0 & 1
    \end{bmatrix}.
\]
Therefore, the cocycle condition \emph{fails} over $\SLC$ when $\delta=1$,
even though it still holds after projecting the labelling elements to $\PSLC$.
The same situation occurs when both co-orientations point away from the tetrahedron:
this is seen by replacing $S$ by $-S$ in the above calculation.

\Proofpart{Step 4:}{Relationship with $\rho$.}
Let $\rho$ be the boundary-parabolic $\PSLC$ representation determined, up to conjugacy,
by the shape parameter solution $z$ for which $W$ is a strong combinatorial flattening.
It remains to justify that the \nword{\PSLC}{cocycle} induced from $\mathcal{L}$
recovers the conjugacy class of $\rho$.

A direct construction of a \nword{\PSLC}{cocycle} on $\triang_{00}$ from the hyperbolic shapes
is given in \cite[\S2.3.2]{siejakowski-thesis}.
Therefore, it suffices to show that these two \nword{\PSLC}{cocycles} differ by a coboundary.

We denote by $\mathcal{L'}$ the cocycle of \cite{siejakowski-thesis}, the definition
of which we now briefly recall:
\begin{itemize}
    \item
    $\mathcal{L'}$ labels all long edges in $\triang_{00}$
    with the complex M\"obius transformation $w\mapsto w^{-1}$
    represented by the \nword{\SLC}{matrix}
    $R=\left[\begin{smallmatrix}0&i\\i&0\end{smallmatrix}\right]$.
    \item
    Medium edges are labelled with the M\"obius transformation $w\mapsto 1-w$,
    which is the image of $C=\left[\begin{smallmatrix}-i&i\\0&i\end{smallmatrix}\right]$
    under the quotient homomorphism.
    \item
    $\mathcal{L'}$ labels each oriented short edge with $\pm H_{Z(\square)}\in\PSLC$.
\end{itemize}
Note that the M\"obius transformations induced by $R$ and $C$ are both involutions, so the choice of
orientations of the long and medium edges is not needed in the construction of $\mathcal{L'}$.
For this reason, $\mathcal{L'}$ is defined purely in terms of
the shape parameters, without the need for a flattening.

We now explain how to obtain $\mathcal{L}$ from $\mathcal{L'}$.
Suppose that a strong combinatorial flattening $W$ for $z$ has been chosen
and that the \nword{1}{cells} of $\triang_{00}$ have been oriented accordingly.
We shall call a vertex of $\triang_{00}$ a \emph{sink} if it is incident to a medium edge oriented towards it.
For each sink vertex, we perform a coboundary action on $\mathcal{L'}$, locally conjugating by
the M\"obius transformation $w\mapsto -w$.
Hence, the M\"obius transformation labelling the adjacent medium edge becomes
\[
    \bigl(w \mapsto -(1-w)\bigr) = \bigl(w \mapsto w-1\bigr) \sim
    \pm\begin{bmatrix}1 & -1\\0 & 1\end{bmatrix} = \pm T.
\]
Similarly, the long edge adjacent to a sink vertex, labelled in $\mathcal{L'}$ by the
element $w\mapsto w^{-1}$, will conjugate to
\[
    \bigl(w \mapsto (-w)^{-1}\bigr)
    =
    \bigl(w \mapsto \tfrac{-1}{w}\bigr)
    \sim \pm\begin{bmatrix}0 & -1\\1 & 0\end{bmatrix}=\pm S.
\]
This coboundary modification will also modify the elements $\pm H_{Z(\square)}$ associated to the short
edges in $\mathcal{L'}$.
It is easy to see that this modification amounts to shifting the log-parameters~$Z$
by integer multiples of $\pi i$,
which has the same effect as replacing them by
the values of the flattening~$W$.
\end{proof}

\begin{corollary}\label{cor:almost-SLC}
As seen in Step~3 of the above proof, the labelling $\mathcal{L}$
satisfies the $\SLC$ cocycle conditions around all \nword{2}{cells} of $\triang_{00}$
except possibly for the rectangles.
Moreover, the holonomy of $\mathcal{L}$ around a rectangle with the parity parameter $\delta=\delta(\square)$
is always $(-1)^\delta \left[\begin{smallmatrix}1&0\\0&1\end{smallmatrix}\right]$.
\end{corollary}

\begin{proof}[Proof of Theorem~\ref{rect=Obs}]
We begin by showing part~\eqref{item:periph-Obs}.
Suppose that $z$ is an algebraic solution of edge consistency and completeness
equations on $\triang$ and that $W$ is a strong combinatorial flattening for $z$.
Denote by $\rho=\rho_z$ the associated boundary-parabolic representation $\pi_1(M)\to\PSLC$.
Using $W$, we construct a boundary-compatible labelling $\mathcal{L}$.
By Corollary~\ref{cor:almost-SLC}, the obstruction class to lifting $\rho$
to a boundary-unipotent \nword{\SLC}{representation} is represented by the cocycle
$\rect(\delta) \in Z^2(\triang_{00},\partial\triang_{00};\FF_2)$,
where $\delta\in\FF_2^{Q(\triang)}$ is given in terms of $W$ by equation~\eqref{def-parity-of-flattening}.
Observe that $\delta(\square)+\delta(\square')+\delta(\square'')=1$ and that $\delta$ has
even parity along all unoriented curves in normal position with respect to $\triang$.
Therefore, $\delta\in\AFpar(\triang)$.

We must now show that $(-1)^\delta$ belongs to the same connected component of $\SAS_0(\triang)$
as the \Sval\ angle structure $\omega(\square) = z(\square)/|z(\square)|$.
Recall that $\delta(\square)=[f(\square)]_2$ where $\pi i f(\square) = Z(\square) - W(\square)$.
By Definition~\ref{def:flattening}, $\Imag W\in\TAS_0(\triang)$, i.e., $\Imag Z - \pi f\in\TAS_0(\triang)$.
Since $\exp\bigl(i\Imag Z(\square)\bigr) = \omega(\square)$,
Remark~\ref{TAS-and-connected-components} implies that
the circle-valued angle structures $\omega$ and $(-1)^\delta$ lie in the same connected
component of $\SAS_0(\triang)$.

In order to prove part~\eqref{item:noperiph-Obs}, we must consider an algebraic solution $z$
of the edge consistency equations which does not necessarily satisfy the completeness equations.
Let $\rho$ be the associated $\PSLC$ representation and let $\omega(\square)=z(\square)/|z(\square)|$
be the \Sval\ angle structure determined by $z$.
Once again, we may choose log-parameters $Z(\square)$ such that
$Z(\square) + Z(\square') + Z(\square'') = i\pi$ in every tetrahedron.
Let $c\in2\ZZ^{\edges(\triang)}$ be the even vector given by
$\mathbf{G}Z=\pi i c$.
By Theorem~\ref{generalization-of-6.1}, there exists a vector $\eta\in\ZZ^{Q(\triang)}$
such that $\eta(\square) + \eta(\square') + \eta(\square'')=1$ in every tetrahedron,
$\mathbf{G}\eta = c$, and moreover $\eta$ has even parity along all curves in normal
position with respect to $\triang$.
Therefore, $W(\square) := Z(\square) - i\pi\eta(\square)$
satisfies all of the conditions in the definition of a strong combinatorial
flattening except for the vanishing of $\mathbf{G_\del}W$.

Using Rule~\ref{labelling-rule}, we may produce a labelling $\mathcal{L}$ of \nword{1}{cells}
of $\triang_{00}$ with elements of $\SLC$ in terms of $W$.
Since the representation $\rho$ is not assumed boundary-parabolic,
the labelling $\mathcal{L}$ will not induce boundary-unipotent representations
of the peripheral subgroups.
Nonetheless, Steps~1, 3 and~4 in the proof of Lemma~\ref{labelling-works} do not
require this assumption and therefore can be repeated without modification.
In particular, for $\delta=[\eta]_2$ we have $\delta\in\AFpar(\triang)$ and
$[\rect(\delta)]=\Obs(\rho)\in H^2(M;\FF_2)$.
Moreover, $(-1)^\delta$ lies in the same connected component of $\SAS(\triang)$ as $\omega$,
because $\Imag W \in \TAS(\triang)$.
\end{proof}

\section{Example: the sister of the figure-eight knot complement}\label{sec:example-m003}
In this section we will illustrate our theory by examining the case of the standard two tetrahedron triangulation~$\triang$ of the ``sister'' of the figure-eight knot
complement~$M$, denoted $m003$ in SnapPy~\cite{snappy}. See Figure~\ref{m003tri}. This manifold has 
$H^2(M,\del M; \FF_2)=\FF_2$ and $H^2(M;\FF_2)=0$,
so, according to our Theorem~\ref{intro-connected-component-counts}, the space 
$\SAS(\triang)$ is a single 3-dimensional torus, while $\SAS_0(\triang)$ consists of two 1-dimensional tori. 

From G\"orner's Ptolemy database of boundary-parabolic representations \cite{ptolemy}, we see that $m003$ has four (conjugacy classes of) 
boundary parabolic $\PSLC$-representations. One is the complex representation $\rho_1$ associated to the complete hyperbolic structure, one is the complex conjugate of $\rho_1$, and there are two real representations. Below we will consider one of these real representations and denote it $\rho_2$; the other is its Galois conjugate. We see from \cite{ptolemy} that the obstruction classes associated to $\rho_1$ and $\rho_2$ are the two different elements of $H^2(M,\del M; \FF_2)=\FF_2$. 
It is known that the class associated to the complex representation $\rho_1$ is non-trivial (see, e.g. \cite[Corollary~2.4]{calegari}), and hence that the class associated to the real representation $\rho_2$ is trivial.

Thus our main theorem, Theorem~\ref{obstruction-theorem}, predicts that if we take solutions of Thurston's equations determining these two representations, then their corresponding \Sval\ angle structures $\omega_1$ and $\omega_2$ will lie on different components of $\SAS_0(\triang)$, and moreover that $\Phi_0(\omega_1)= 1$, while $\Phi_0(\omega_2)=0$ in $H^2(M,\del M; \FF_2)$. 
To illustrate our theory we will  concretely verify 
these predictions.

\begin{figure}[ht]
\includegraphics[width=0.95\columnwidth]{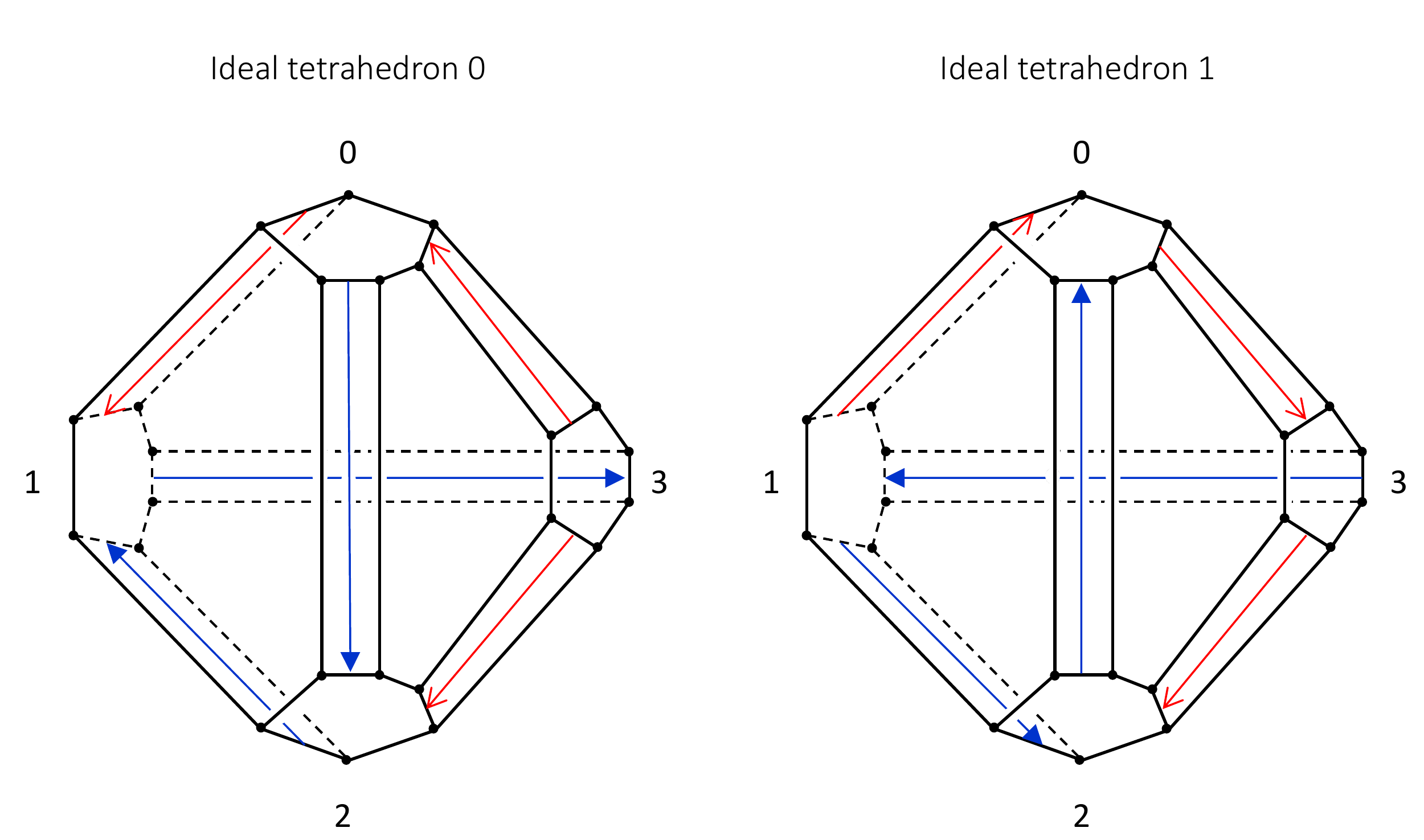}
\caption{\label{m003tri} Diagram of the doubly truncated tetrahedra arising from the ideal triangulation of m003, with vertices labelled as in SnapPy. Red arrows (marked with a thin arrow-head) are parallel to the edge 0, while blue arrows (marked with a full arrow-head) are parallel to the edge 1. Face pairings take each face of ideal tetrahedron 0 to the corresponding face of ideal tetrahedron 1, by the reflection preserving edge colourings.}
\end{figure}

\medskip
The full gluing matrix is given by SnapPy as
\[
    \begin{bmatrix}\mathbf{G}\\\mathbf{G_\del}\end{bmatrix} =
    \begin{bmatrix}G & G' & G''\\G_\del & G'_\del & G''_\del\end{bmatrix} =
    \left[\begin{array}{cc|cc|cc}
        2 & 2 & 0 & 0 & 1 & 1 \\
        0 & 0 & 2 & 2 &  1 & 1 \\
        0& 2 & -2 & 0 & 0 & 0 \\
        0& 2& -1 & -1 & 0 & 0
    \end{array}\right].
\]
This matrix corresponds to the standard ordering of quad-types in each tetrahedron of Figure~\ref{m003tri}: $\{01,23\}$, $\{02,13\}$, $\{03,12\}$. 
The matrix of leading-trailing deformations based on the two edges of $\triang$ is
\[
    \mathbf{L}=
    \begin{bmatrix}L & L' & L''
    \end{bmatrix} =
    \left[\begin{array}{cc|cc|cc}
        1 & 1 & 1 & 1 & -2 & -2 \\
        -1 & -1 & -1 & -1 &  2 & 2 
    \end{array}\right].
\]

First we'll consider the complex representation $\rho_1\colon \pi_1(M) \to \PSLC$. The geometric solution of Thurston's gluing equations giving the complete
hyperbolic structure of finite volume is given by $z_1=z_2=e^{i\pi/3}$; this gives a boundary-parabolic representation $\rho_1\colon \pi_1(M) \to \PSLC$. The associated $S^1$-valued angle structure is 
$$\omega_1= e^{i \pi/3}(1,1,1,1,1,1).$$
According to our definition of $\Phi_0$ (see Section~\ref{obstruction-map-defn}), we need to find a $\{\pm 1\}$-valued angle structure on the same component of $\SAS_0(\triang)$ as $\omega_1$ satisfying a  certain parity condition, and apply the rectangle map to that. Such a $\{\pm 1\}$-valued angle structure can be obtained from a strong combinatorial flattening of the solution $z$, as we discuss in Section~\ref{combinatorial_flattenings_section}.

We find there is a strong combinatorial flattening given by the integer adjustment vector $f_1=(0,0,0,0,1,1)$ with corresponding  mod 2 parity vector $\delta_1=(0,0,0,0,1,1) \bmod 2$.  
(Here it is enough to check the parity condition in Definition~\ref{strong-flattening} holds for edge loops and peripheral curves since inclusion induces a surjection $H_1(\del M;\FF_2) \to H_1(M;\FF_2)$.) 
Then $(-1)^{\delta_1} = (1,1,1,1,-1,-1)$ is a $\{\pm 1\}$-valued angle structure also lying in the geometric component of $\SAS_0(\triang)$.
Now $\Phi_0(\omega_1)\in H^2(M,\del M;\FF_2)$ is defined to be the 
cohomology class represented by $\rect(\delta_1)$. This cocycle $\rect(\delta_1)$ is 
supported on the shaded faces on the complex in Figure~\ref{rect1}. 

We expect that the class $[\rect(\delta_1)]\neq 0\in H^2(M,\del M;\FF_2)$. A brief elementary argument for this is as follows. In Figure~\ref{rect1} we show that this class is represented by a cocycle which takes the value 0 everywhere except on the internal hexagonal faces 012 and 023, where it takes the value 1. This cocycle cannot be the coboundary of a combination $y$ of the internal edges (``internal" here means the edges of the complex not lying in the boundary). The reason is that if such a coboundary is zero on every rectangle, as required, then $y$ must be either $0$ on every edge of the complex parallel to a red arrow, or $1$ on every such edge. The same goes for the edges parallel to the blue arrows. Thus there are only 4 such $y$ which need to be checked explicitly, and none of them give this cocycle.  
Thus $[\rect(\delta_1)] \neq 0$ so $\Obs_0(\rho_1)=[\rect(\delta_1)]\neq 0 \in H^2(M,\del M;\FF_2)$ and the case $\rho_1$ checks out.
 
 \begin{figure}[ht]
 \includegraphics[width=0.95\columnwidth]{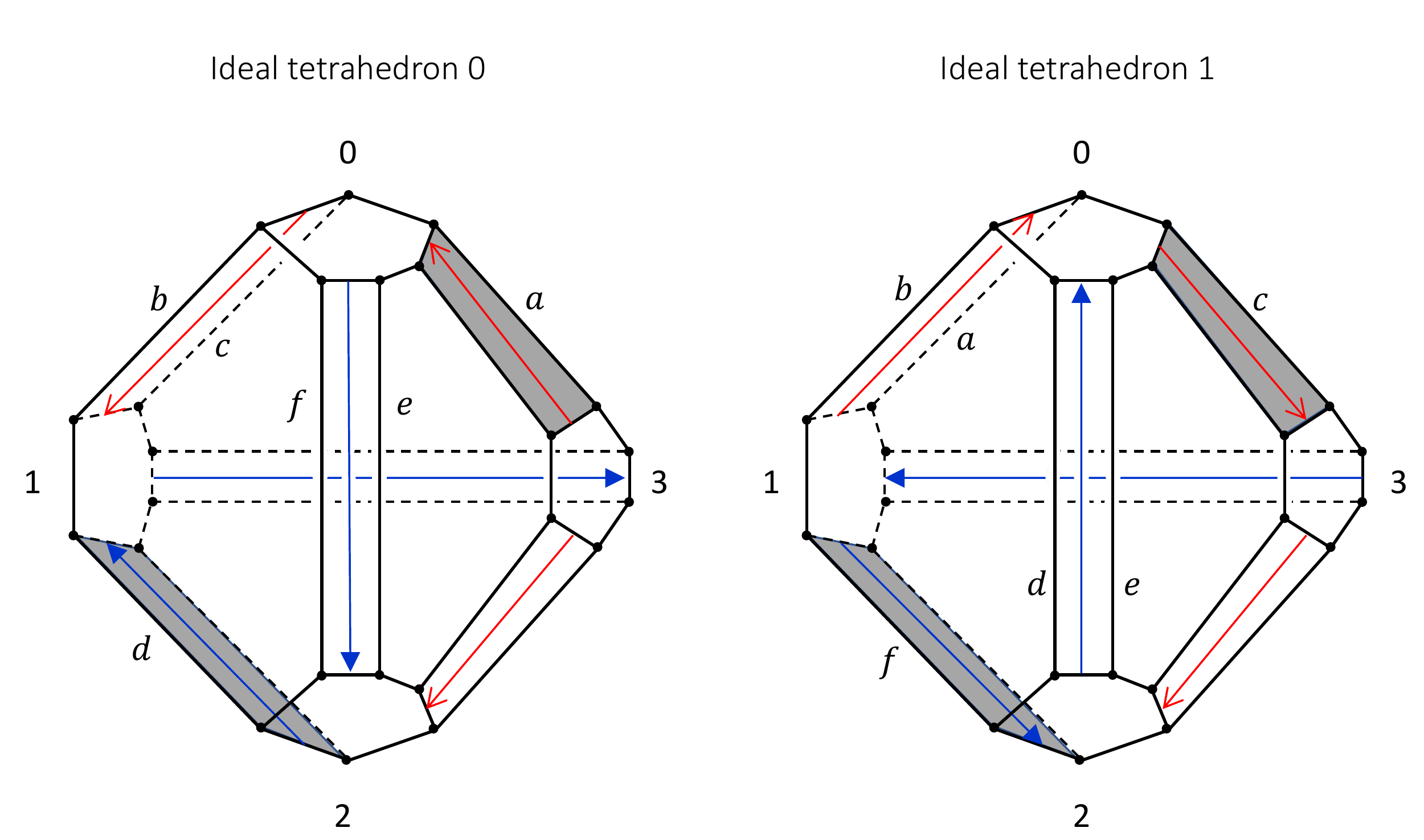}
 \caption{\label{rect1} The 2-cocycle $\rect(\delta_1)$ is supported on the shaded rectangles. Adding the coboundary of 1-chain $a+b+c+d+e+f$ supported on the labelled edges gives a 2-cocycle $c'$ supported on the 2 internal hexagonal faces 012 and 023. An elementary analysis verifies that this cocycle is not the coboundary of any 1-cochain supported on the internal edges. 
 }
 \end{figure}

\medskip

Next we'll consider the real boundary parabolic representation $\rho_2\colon \pi_1(M) \to \PSLR$. This arises from 
an interesting algebraic solution of Thurston's gluing and completeness equations given by
$z_1=z_2 = \tau$, 
where $\tau = (1 + \sqrt{5})/2$ is the golden ratio. 
The associated $S^1$-valued angle structure is $$\omega_2=(1,1,-1,-1,1,1).$$
Note that this angle structure can be expressed as $(-1)^{\delta_2}$
where $\delta_2=(0,0,1,1,0,0)$. 
Since $\omega_2$ is already a $\ZZ_2$-angle structure satisfying the parity conditions,  $\Phi_0(\omega_2)$ is represented directly by the cocycle $[\rect(\delta_2)]$. 
This cocycle is shown is Figure~\ref{rect2}. A quick calculation described there shows this class is zero in $H^2(M,\del M;\FF_2)$.
Thus $\Obs_0(\rho_2)=[\rect(\delta_2)]=0\in H^2(M,\del M;\FF_2)$ and the case $\rho_2$ also checks out.

 \begin{figure}[ht] 
  \includegraphics[width=0.95\columnwidth]{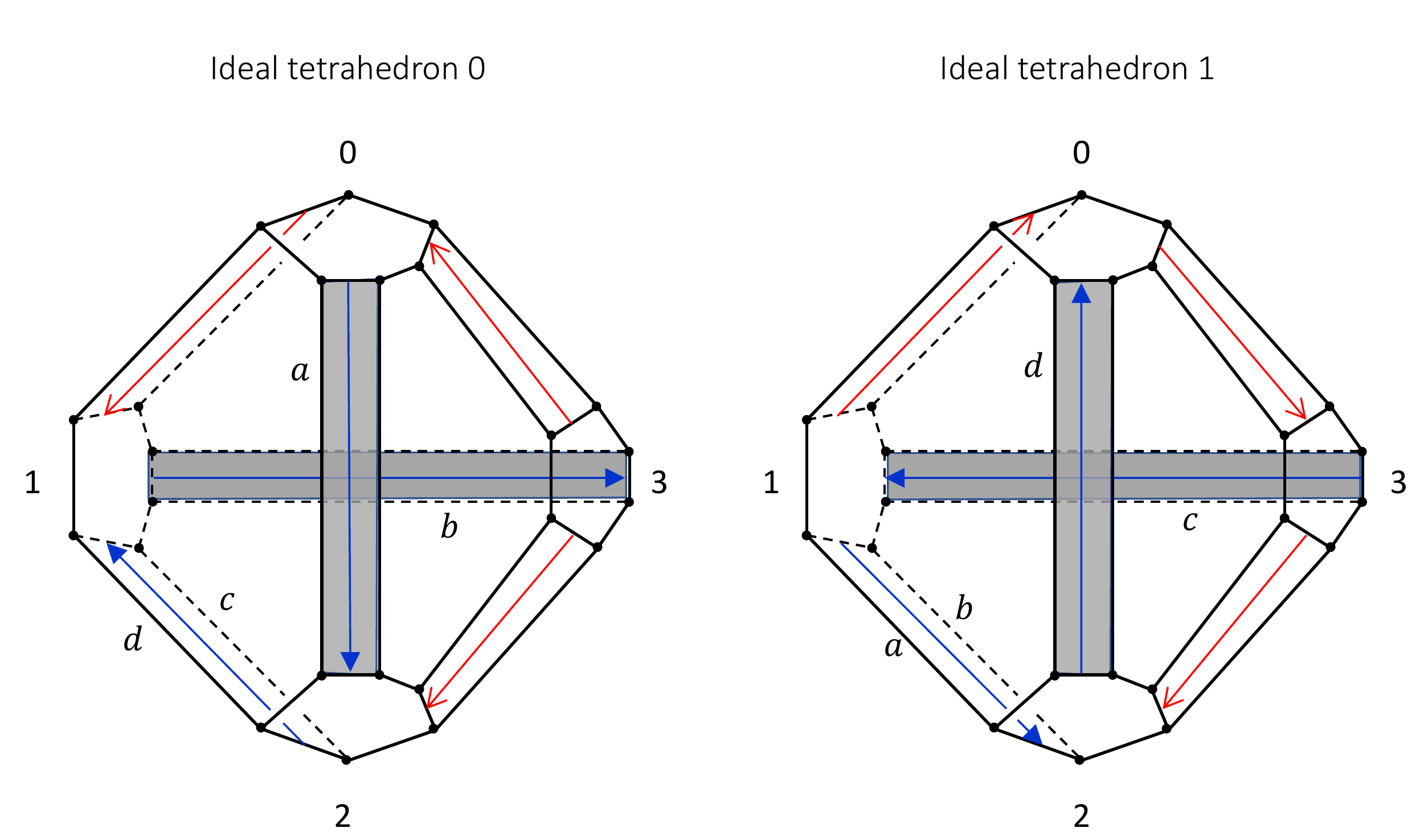}
 \caption{\label{rect2} The 2-cocycle $\rect(\delta_2)$ is supported on the shaded rectangles. Adding to this cocycle the coboundary of the 1-chain $a+b+c+d$ obtained from  labelled edges gives the zero 2-cocycle. Hence $[\rect(\delta_2)]$ is zero in relative cohomology.}
 \end{figure}

It is also straightforward to check from these details that the two angle structures $\omega_1$ and $\omega_2$ do indeed lie on different components of $\SAS_0(\triang)$. Indeed, because the leading-trailing deformations based on the edges generate the tangent space of $\SAS_0(\triang)$, if $\omega_1$ and $\omega_2$ lay on the same component, there would be a real number $x$ such that
\[
\delta_1-\delta_2 = (0,0,0,0,1,1)-(0,0,1,1,0,0)=x(1,1,1,1,-2,-2) \bmod 2 .
\]
But there is no such $x$.

%==============================================================================
% The bibliography
%==============================================================================

%==============================================================================
% The End
%==============================================================================
\end{document}